\documentclass[12pt]{article}
\usepackage{paralist,amsmath,amsthm,amssymb,mathtools,url,graphicx,pdfpages,tikz,pdfpages,rotating}

\usetikzlibrary{positioning,chains,fit,shapes,calc}
\usetikzlibrary{arrows,automata}
\usepackage[affil-it]{authblk}
\usepackage[left=1in,right=1in,top=1.2in, bottom=1in]{geometry}
\usepackage{mathdots}
\usepackage{bm} 
\usepackage{indentfirst}
\usepackage{xcolor}
\usepackage{graphicx}
\usepackage[
    colorlinks=true,
    citecolor=blue,   
    linkcolor=black,   
    urlcolor=blue,    
    filecolor=blue    
]{hyperref}

\newtheorem{theorem}{Theorem}[section]
\newtheorem{lemma}[theorem]{Lemma}

\theoremstyle{definition}
\newtheorem{definition}[theorem]{Definition}
\newtheorem{example}[theorem]{Example}
\newtheorem{observation}[theorem]{Observation}

\numberwithin{equation}{section}

\renewcommand{\textcolor}[3][]{%
  \color{black}#3%
}
\renewcommand{\color}[2][]{}

\DeclareMathOperator{\vol}{vol}
\renewcommand{\i}{\rm i}
\renewcommand{\Im}{\rm Im}
\renewcommand{\Re}{\rm Re}

\DeclareMathOperator{\rank}{rank}

\newcommand{\R}{\mathbb{R}}
\newcommand{\C}{\mathbb{C}}

\title{Combinatorial formula for the Moore--Penrose inverse of the complex signless Laplacian of an oriented graph
}

\author{XiaoYang Liu\footnote{University of California, Berkeley, 110 Sproul Hall, Berkeley, California 94720 (xiaoyangliu@berkeley.edu).}, Sudipta Mallik\footnote{Department of Mathematics and Physics, Marshall University, 1 John Marshall Drive, Huntington, West Virginia 25755 (mallik@marshall.edu).},
and Anh Tang\footnote{Georgia Gwinnett College, 1000 University Center Lane, Lawrenceville, Georgia 30043 (atang1@ggc.edu).}
}
\date{}

\begin{document}

\maketitle

\begin{abstract}
We find necessary and sufficient conditions for the $\rank$ of the signless incidence matrix of a weakly connected oriented graph with non-zero complex edge weights. We use this to find the combinatorial formulas for the Moore--Penrose inverse of the complex signless incidence and complex signless Laplacian matrix of a weakly connected oriented graph with non-zero complex edge weights. This resolves the open problem posed in the concluding remarks of Barik et. al. (Discrete Mathematics 349 (9), 115117, 2026). 
\end{abstract}

\section{Introduction}

\textcolor{blue}{
\indent In 1965, the Moore-Penrose inverse was first studied in a combinatorial setting, specifically focusing on oriented incidence matrices of graphs \cite{I}. Then, Bapat and others studied the Laplacian and oriented incidence matrices of certain  classes of graphs such as trees and distance regular graphs \cite{bap,AB,ABE}. Research then extended to the study of the signless Laplacian of graphs by \cite{CRC,HM}, and Hessert and Mallik furthered studied the Moore-Penrose inverse of the incidence matrix and signless Laplacian of trees and unicyclic graphs \cite{HM1,HM2}. Work was later done on the incidence matrix and signless Laplacian of wheel graphs and bipartite graphs \cite{Ipsen,Milica1}.  Alazemi, Andeli\'c, and Mallik further extended this field by studying the more generalized setting of signed graphs \cite{Milica2}.
\\
\indent Recently, \cite{MP2026} found the combinatorial formulas for the Moore-Penrose inverse of the complex incidence matrix and complex Laplacian matrix of a weakly connected oriented graph with complex edge weights. They then asked the open question (viz. Question 5.1 in \cite{MP2026}) of finding the combinatorial formulas for the Moore-Penrose inverse of the complex signless incidence matrix and complex signless Laplacian matrix of a weakly connected oriented graph with complex edge weights which will be the objective of this article.
}

\textcolor{blue}{
An oriented graph is a specific type of directed graph where the underlying undirected graph— that is, the graph left when one removes the orientation of every edge— is a simple graph. A weakly connected oriented graph is a graph where 
the underlying graph is a simple and connected graph where every edge is assigned an orientation.
\\
\indent Let $G$ be a weakly connected oriented graph on $n$ vertices $V(G) = \{ 1, 2,...,n \} $ and $m$ edges $E(G) = \{ e_1, e_2,...,e_n \}$ with its underlying undirected graph $G_u$. Suppose further that the directed edge $(u,v)$, from vertex $u$ to vertex $v$, of $G$ has a non-zero complex weight denoted by $w_{uv}$. We define the complex
adjacency matrix, the complex signless incidence matrix, and the complex signless Laplacian matrix of $G$ using the following definitions.}

\begin{definition}
    \textcolor{blue}{The complex adjacency matrix of $G$ is the complex $n \times n$ matrix $A = (a_{ij})$ whose rows and columns are both indexed by the vertices of $G$ and is defined as:
    $$a_{ij}= 
    \begin{cases}
        w_{ij} & \text{if $(i,j) \in E(G)$} \\
        \overline{w_{ij}} & \text{if $(j,i) \in E(G)$} \\
        0 & \text{otherwise} \\
    \end{cases}
    $$}
\end{definition}
  
\begin{definition}
    \indent \textcolor{blue}{The complex signless incidence matrix of $G$ is the complex $n \times m$ matrix $N = (n_{ij})$ with rows indexed by the vertices of $G$ and whose columns are indexed by the edges of $G$, and is defined as:}
$$
\textcolor{blue}{n_{ij} = 
\begin{cases} 
\sqrt{w_{uv}} &  \text{if } i = u, \\ 
\sqrt{\overline{w_{uv}}} & \text{if } i = v, \\ 
0 & \text{otherwise} 
\end{cases}}
$$ 
where $e_j = (u, v)$ and  $\sqrt{w_{uv}}$ denotes the principal square root of  $w_{uv}$.
\end{definition}

\begin{definition}
    \color{blue} {The complex signless Laplacian matrix of $G$ is the complex $n \times n$ matrix $Q$ such that $Q = D + A$, where $D = (D)_{ij}$ is the diagonal matrix of order $n$ with $(D)_{ii} = \sum_{j \neq i} \lvert a_{ij} \rvert$. It can be shown that $Q = NN^*$ where $N^*$ is the conjugate transpose of $N$.}
\end{definition}

A walk,$\gamma$, in the underlying graph, $G_u$, is a sequence of vertices and edges where the set of edges traversed by $\gamma$ is denoted $E(\gamma)$ and the set of vertices is denoted $V(\gamma)$. Note that $\gamma$ may not repeat vertices and therefore edges. We follow the definition of the weight of a walk in the underlying connected, undirected graph, $G_u$, of a weakly connected oriented graph $G$ used in \cite{MP2026}. 
\begin{definition}
    Suppose $G$ is a weakly connected oriented graph on $n$ vertices $1,2,..,n$ and let $w_{uv}$ denote the weight of the directed edge $(u,v) \in E(G)$. For a walk in $G_u$, denoted as $\gamma = (i_1,i_2,...,i_k)$, we will define the weight of $\gamma$ to be 
    \[
    wt(\gamma) = \prod_{s = 1}^{k-1}a_{i_si_{s+1}} 
    \]
    where $a_{i_si_{s+1}}$ are the entries of the complex adjacency matrix defined above. 
    \\
    \indent The unit weight of $\gamma$ is defined as $\frac{wt(\gamma)}{|wt(\gamma)|}$. 
    \\
    \indent Note that the weight of a cycle $C$ in $G$ is the weight of any closed walk in $G_u$ which has 2 possible values, one weight is given by taking the clockwise orientation and the other is obtained by taking the counterclockwise direction. 
    \\
    \begin{observation}\label{treeunitweight}
        When $\gamma_{i \rightarrow j}$ is a walk from vertex $i$ to vertex $j$ in an underlying graph that is a tree, there is only one possible value for the unit weight, namely $\frac{wt(\gamma'_{i \rightarrow j})}{|wt(\gamma'_{i \rightarrow j})|}$ where $\gamma'_{i \rightarrow j}$ is the unique walk from vertex $i$ to vertex $j$ in the underlying tree.
    \end{observation} 
\end{definition}

\begin{definition}
\textcolor{blue}{
    $W_G$ is defined as the product of all edge weights of an oriented graph $G$, and $|W_G|$ denotes the norm of $W_G$. 
    }
\end{definition}

\begin{definition}
    \textcolor{blue}{Following the standard definition, the Moore-Penrose Inverse of any $n \times m$ matrix $M$ is the unique $m \times n$ matrix, denoted as $M^+$, that satisfies all four criteria below:
    \[
    MM^+M = M , M^+MM^+ = M^+, (MM^+)^* = MM^+, (M^+M)^* = M^+M
    \]
    }
\end{definition}

\begin{observation}\label{negreal}
    \textcolor{blue}{
    For any edge $e_i = (u,v) \in E(G)$, we restrict the corresponding weight to obey the following condition: $w_{uv} \in \C \setminus \R^- \setminus \{0\}$. Throughout the following proofs, we frequently use the assumption that $\sqrt{a}\cdot\sqrt{\overline{a}} = |a|$ which is true if and only if $a$ is not a negative real number, where the square root is taken to be the principle square root. Thus, throughout the rest of the paper, we will restrict all edge weights to be non-zero values in $\C \setminus \R^-$.
    }
\end{observation}

\textcolor{blue}{
In Section 2, we will find the rank of the complex signless incidence matrix $N$ and establish some properties of $N^+$. Section 3 will provide the combinatorial formulas of the Moore--Penrose inverse of the complex signless incidence matrix for trees; even unicyclic graphs where the weight of the cycle is not a positive real number; and odd unicyclic graphs where the weight of the cycle is not a negative real number. In Section 4, we will study the volume of $N$ for all cases of $G$ and provide the combinatorial formulas for the Moore--Penrose inverse of the complex signless incidence matrix of an arbitrary weakly connected oriented graph. Section 5 will conclude with a study of the determinant of $Q$ and provide the combinatorial formulas for the Moore--Penrose inverse of the complex signless Laplacian matrix of a weakly connected oriented graph.
}

\section{Finding Rank of N}

\begin{theorem} \label{treerank}
    \textcolor{blue}{If $G$ is a oriented tree on $n \geq 2$ vertices with non-zero edge weights in $\C \setminus \R^-$ and complex signless incidence matrix $N$, then $\rank(N) = n-1$}. 
\end{theorem}
\begin{proof}
    \textcolor{blue}{Since G is a tree, there exists at least 2 leaves and $N$ has dimensions $n \times (n-1)$. Suppose, without loss of generality, that $e_j = (u,v) \in E(G)$ is a leaf. If we consider the equation $\sum_{k=1}^{n-1} c_k n_k = 0$ where $n_k$ corresponds to the kth column of $N$ and $c_k \in \C$, the entry in row $v$ will simply be $c_j\sqrt{\overline{w_{uv}}}=0$. This implies that $c_j=0$, and therefore, we can remove row $v$ and col $j$ from $N$ to obtain an $(n-1) \times (n-2)$ matrix which define as $N'$. Since taking away edge $j$ and vertex $v$ still leaves a tree, $N'$ is still the signless incidence matrix of a tree, and by induction, we conclude that $c_k=0  \text{ for all } k \in {1,2,...,n-1}$. Therefore, the columns of $N$ are linearly independent, and so $\rank(N) = n-1$}.
\end{proof}

\begin{lemma} [\cite{MP2026}, \textit{Lemma 2.3}] \label{barik}
    \textcolor{blue}{
    Let $G$ be a weakly connected oriented graph on $n$ vertices with complex edge weights. Then $wt(C)$ is a positive real number for each oriented cycle $C$ in $G$ if and only if $\frac{wt(\gamma_1)}{|wt(\gamma_1)|}=\frac{wt(\gamma_2)}{|wt(\gamma_2)|}$ for any two walks $\gamma_1$ and $\gamma_2$ from a fixed vertex to another fixed vertex in $G_u$. 
    }
\end{lemma}

\begin{lemma}\label{shortestwalk}
    \textcolor{blue}{Let $G$ be a weakly connected oriented graph on $n \geq 2$ vertices, $m \geq 1$ edges, and non-zero edge weights in $\C \setminus \R^-$ such that $wt(C) \in \R^+$ for all even cycles. Let $u,v \in V(G)$. If $\gamma_{1}$ and $\gamma_{2}$ are both the shortest walk from u to v, then $\frac{wt(\gamma_{1})}{|wt(\gamma_{1})|} = \frac{wt(\gamma_{2})}{|wt(\gamma_{2})|}$}.
\end{lemma}

\begin{proof}
    \color{blue}{Let $\gamma_{1}$ and $\gamma_{2}$ both be shortest walks from $u$ to $v$. Then, the edges not shared by the two walks form oriented cycles $C_i$, where $\gamma_1[C_i]$ and $\gamma_2[C_i]$ are the set of vertices and undirected edges of $C_i$ traced out by $\gamma_1$ and $\gamma_2$, respectively. All $C_i$ must be cycles of even length for if not, suppose they formed at least 1 odd cycle $C_o$ where $len(\gamma_1[C_o])>$ $len(\gamma_2[C_o]$) ($len$ denotes the length of the walk). Then we can take $\gamma_3$ to be a walk from $u$ to $v$ that follows $\gamma_1$ for all edges not in $C_o$ but takes the walk given by $\gamma_2[C_o]$ in $C_o$. Then $\gamma_{1}$ is not the shortest walk.  
    Then we have:
    $$\frac{wt(\gamma_{1})}{wt(\gamma_{2})} = \prod_i\frac{wt(\gamma_1[C_i])}{wt(\gamma_2[C_i])} = \prod_i \left( \frac{wt(\gamma_1[C_i])}{wt(\gamma_2[C_i])}\right) \frac{\overline{wt(\gamma_2[C_i]})}{\overline{wt(\gamma_2[C_i]})}
    = \prod_i\frac{wt[C_i]}{|wt(\gamma_2[C_i])|^2} \in \R^+
    $$
    Therefore, we have $\frac{wt(\gamma_{1})}{|wt(\gamma_{1})|} = \frac{wt(\gamma_{2})}{|wt(\gamma_{2})|}$.
    }
\end{proof}

A walk,$\gamma$, in $G_u$ is a sequence of vertices and edges where the set of edges traversed in $\gamma$ is denoted $E(\gamma)$ and the set of vertices is denoted $V(\gamma)$.
\begin{lemma}\label{dis}
    \textcolor{blue}{Let $G$ be a simple, connected  graph on $n \geq 2$ vertices and $m \geq 1$ edges where $(u,v) \in E(G)$ and $i \in V(G)$. If $d(i,u) = d(i,v)$ and $\gamma_{i \rightarrow u}$ is a shortest walk from vertex $i$ to vertex $u$, then there exists a shortest walk from vertex $i$ to vertex $v$, denoted as $\gamma'_{i \rightarrow v}$, such that $\gamma'_{i \rightarrow v} \cup \gamma_{i \rightarrow u} \cup 
    \{u,v\}$ forms an odd unicyclic graph in $G$.} 
\end{lemma}

\begin{proof}
    \textcolor{blue}{Let $d(i,u) = d(i,v)$, $\gamma_{i \rightarrow u}$ be a shortest walk from vertex $i$ to vertex $u$, and $\gamma_{i \rightarrow v}$ be any shortest walk from vertex $i$ to vertex $v$. 
    In the undirected graph formed by $S = \gamma_{i \rightarrow u} \cup \gamma_{i \rightarrow v}$,  let $w_1,w_2,..,w_z$ be all the branch vertices in $S$. A branch vertex in $S$ is defined as any vertex with degree strictly greater than 2.  
    Note that $z \geq 1$ because if $z = 0$, that implies that $\gamma_{i \rightarrow u}$ and $ \gamma_{i \rightarrow v}$ never diverge, and thus $u = v$ which contradicts the assumption that $G$ is a weakly connected graph. 
    \\
    \indent Also, since $\gamma_{i \rightarrow u}$ and $ \gamma_{i \rightarrow v}$ are both shortest walks, any subwalk of these two walks must themselves be a shortest walk. So, if we let $\gamma_{i \rightarrow u}(w_p \rightarrow w_q)$ be the subwalk of $\gamma_{i \rightarrow u}$ from vertex $w_p$ to vertex $w_q$, then we have:
    $$len[\gamma_{i \rightarrow u}(1 \rightarrow w_{z})] =  len[\gamma_{i \rightarrow v}(1 \rightarrow w_{z})]$$
    \indent Thus, if we take $\gamma'_{i \rightarrow v}$ to be the walk that takes the subwalk $\gamma_{i \rightarrow u}(1 \rightarrow w_{z})$ and then the subwalk $\gamma_{i \rightarrow v}(w_z \rightarrow v)$, then $\gamma'_{i \rightarrow v}$ will be a shortest walk from $1$ to $v$. 
    \\
    \indent By construction, $\gamma'_{i \rightarrow v}$ and $\gamma_{i \rightarrow u}$ only have $w_z$ as a branch vertex so $E(\gamma'_{i \rightarrow v}(w_z \rightarrow v)) \cap E(\gamma_{i \rightarrow u}(w_z \rightarrow u)) = \emptyset$. Since $d(i,u) = d(i,v)$, it follows that $len[\gamma'_{i \rightarrow v}(w_z \rightarrow v)] = len[\gamma_{i \rightarrow u}(w_z \rightarrow u)] \geq 1$. Thus, $\gamma'_{i \rightarrow v}(w_z \rightarrow v) \cup \gamma_{i \rightarrow u}(w_z \rightarrow u) \cup \{u,v\}$ forms an odd cycle.
    \\
    \indent Therefore, $\gamma'_{i \rightarrow v} \cup \gamma_{i \rightarrow u} \cup \{u,v\} = \gamma_{i \rightarrow u}(1 \rightarrow w_{z}) \cup \gamma'_{i \rightarrow v}(w_z \rightarrow v) \cup \gamma_{i \rightarrow u}(w_z \rightarrow u) \cup \{u,v\}$ forms an odd unicyclic graph in $G_u$ as $\gamma_{i \rightarrow u}(1 \rightarrow w_{z})$ contains no cycle. If $\gamma_{i \rightarrow u}(1 \rightarrow w_{z})$ did contain at least 1 cycle, then it wouldn't be a shortest walk because you can choose the walk from $1$ to $w_z$ that doesn't enter any of the cycles which would be an even shorter walk. }
\end{proof}

\begin{example}{
\begin{figure}[htbp]
\centering
\begin{tikzpicture}[
    node distance=2cm and 3cm,
    vertex/.style={
        circle, 
        draw=black, 
        fill=pink!40, 
        minimum size=5mm, 
        inner sep=0pt,
        font=\small\bfseries
    },
    edge/.style={
        ->, 
        >={Stealth[scale=1.2]}, 
        thick, 
        shorten >=1pt
    },
    weight/.style={
        font=\small\color{blue}, 
        fill=white, 
        inner sep=2pt
    }
]

    \node[vertex] (1) at (8,0) {1};
    \node[vertex] (2) at (8,4) {2};
    
    \node[vertex] (4) at (4,0) {4};
    \node[vertex] (3) at (4,4) {3};
    
    \node[vertex] (5) at (1,1.5) {5};
    \node[vertex] (6) at (0,5.5) {6};

    \draw[->] (2) -- node[weight] {1 : $e_2$} (1);
    \draw[->] (1) -- node[weight] {1 : $e_1$} (4);
    \draw[->] (4) -- node[weight] {1 : $e_5$} (3);
    \draw[->] (3) -- node[weight] {1 : $e_3$} (2);
    
    \draw[->] (6) -- node[weight] {1 : $e_8$} (3);
    \draw[->] (3) -- node[weight] {i : $e_4$} (5);
    \draw[->] (5) -- node[weight] {i : $e_7$} (6);
    \draw[->] (5) -- node[weight] {i : $e_6$} (4);

    \end{tikzpicture}
    \caption{\color{blue}{Example of a weakly connected oriented graph such that all three even cycles have positive real weights and all three odd cycles have negative real weights.}}
    \end{figure}
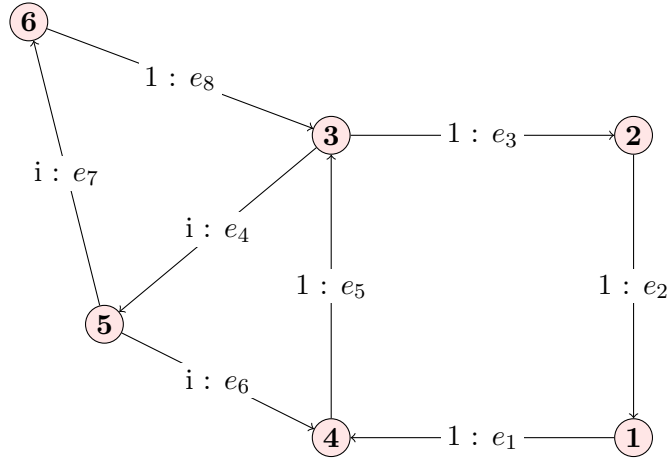   \color{blue}{ For $G$ shown in Figure 1, it's complex signless incidence matrix $N$, is shown below. 
    $$N =
    \begin{bmatrix}
        1 & 1 & 0 & 0 & 0 & 0 & 0 & 0 \\
        0 & 1 & 1 & 0 & 0 & 0 & 0 & 0 \\
        0 & 0 & 1 & \sqrt{\i} & 1 & 0 & 0 & 1 \\
        1 & 0 & 0 & 0 & 1 & \sqrt{-\i} & 0 & 0 \\
        0 & 0 & 0 & \sqrt{-\i} & 0 & \sqrt{\i} & \sqrt{\i} & 0 \\
        0 & 0 & 0 & 0 & 0 & 0 & \sqrt{-\i} & 1
    \end{bmatrix}
    $$
    \\ \indent Using Sage Math, one can confirm that $\rank(N) = 5$, in accordance with Theorem \ref{thm:rank}.}}
\end{example}

\begin{theorem}
\label{thm:rank}
    \textcolor{blue}{Let $G$ be a weakly connected oriented graph on $n \ge 2$ vertices, $m \geq 1$ edges, and non-zero edge weights in $\C \setminus \R^-$ with complex signless incidence matrix $N$. Then, $\rank(N) = n-1$ if and only if $G$ is either a tree, or $wt(C) \in \R^+$ for all even cycles and $wt(C) \in \R^-$ for all odd cycles.}
\end{theorem}
\begin{proof}
    \textcolor{blue}{Let $\rank(N) = n-1$. \\
    \textbf{Case 1:} $G$ is a tree \\
    We are done. \\
    \textbf{Case 2:} Let $C$ be a cycle of length $k$, consisting of vertices ${1,2, ..., k}$ with $(1, k) \in E(G)$. So, there exists a non-zero vector $x$ such that $N^Tx = 0$ where for all $e_j = (u,v) \in E(G)$, $x_u\sqrt{w_{uv}} + x_v\sqrt{\overline{w_{uv}}} = 0$. By Observation \ref{negreal}, since ${w_{uv}}$ is not a negative real number, this implies that $x_{v} = -x_{u}\frac{w_{uv}}{|{w_{uv}}|}$, or equivalently, $x_{v} = -x_{u} \frac{a_{uv}}{|{a_{uv}}|}$. Therefore, if $\gamma= (1,2, ..., k)$, we get: 
    $$x_{k} = (-1)^{k-1}x_1\prod_{i=1}^{k-1}\frac{a_{i,i+1}}{a_{i+1,i}} = (-1)^{k-1}x_1\frac{wt(\gamma)}{|wt(\gamma)|}$$ Also since $(1, k) \in E(G)$, we get 
    $$x_{k} = -x_1 \frac{a_{1,{k}}}{\left| a_{1,{k}} \right|}$$ 
    Equating these two expressions we conclude that 
    $$\left[-x_1 \frac{a_{1,{k}}}{\left|a_{1,{k}} \right|}\right] \frac{\overline{a_{1, {k}}}}{\left|{a_{1, {k}}} \right|}= \left[ -(-1)^kx_1\frac{wt(\gamma)}{|wt(\gamma)|} \right] \frac{\overline{a_{1, {k}}}}{\left|{a_{1, {k}}} \right|}$$ 
    \\Since G is connected, all $ x_m$ are non-zero, which implies:
    $$1 = (-1)^k\frac{{wt(C)}}{|{wt(C)|}}$$
    Thus if C is a cycle of odd length, then $k$ is odd. This means ${wt(C)} = -|{wt(C)}| \in \R^-$. \\
    Thus if C is a cycle of even length, then $k$ is even. This means ${wt(C)} = |{wt(C)}| \in \R^+$. \\ } 
    
    \textcolor{blue}{For the converse direction, we have two cases.
    \\
    \textbf{Case 1}: $G$ is an oriented tree on n vertices.
    \\
    \indent Then by Theorem \ref{treerank}, $\rank(N) = n-1$.
    \\
    \textbf{Case 2}: $G$ has cycles.
    \\
    \indent \textbf{Case 2a:} $G$ has only even cycles. \\
    \indent Then by Lemma \ref{barik}, we can define the vector $x \text{ such that } x_j = (-1)^{d(1,j)}\frac{wt(\gamma_{1\rightarrow j})}{|wt(\gamma_{1\rightarrow j})|}$ where $\gamma_{1\rightarrow j}$ is any walk from $1 \text{ to } j$ and $d(1,j)$ is the distance between vertex $1$ and $j$. Let $e_i = (u,v)$, then: 
        $$(N^Tx)_i = (-1)^{d(1,u)}\sqrt{w_{uv}}\frac{wt(\gamma_{1\rightarrow u})}{|wt(\gamma_{1\rightarrow u})|} +(-1)^{d(1,v)} \sqrt{\overline{w_{uv}}}\frac{wt(\gamma_{1\rightarrow v})}{|wt(\gamma_{1\rightarrow v})|}$$ 
        $$ = (-1)^{d(1,u)}\sqrt{\overline{w_{uv}}} (\frac{w_{uv}}{|{w_{uv}}|}\frac{wt(\gamma_{1\rightarrow u})}{|wt(\gamma_{1\rightarrow u})|} - \frac{wt(\gamma_{1\rightarrow v})}{|wt(\gamma_{1\rightarrow v})|} = 0$$ \\
        Where the second equality holds because $d(1,u) = d(1,v) \pm 1$ by Lemma \ref{dis} and $\frac{w_{uv}}{|{w_{uv}}|}\frac{wt(\gamma_{1\rightarrow u})}{|wt(\gamma_{1\rightarrow u})|}$ is the unit weight of a walk from $1$ to $v$. Therefore, $\rank(N) < n$ and because $G_u$ is a connected graph, there exists a spanning tree in the underlying graph, $G_u$, which implies that there exists a submatrix of N, call it $N'$,  with $\rank(N') =  n-1$. So, $\rank(N) = n-1$. 
    \\
    \indent \textbf{Case 2b:} $G$ has at least 1 odd cycle.
    \\
    \indent By Lemma \ref{shortestwalk}, I can define the vector $x \text{ such that } x_j = (-1)^{d(1,j)}\frac{wt(\gamma_{1\rightarrow j})}{|wt(\gamma_{1\rightarrow j})|}$ where $\gamma_{1\rightarrow j}$ is a shortest possible walk from $1 \text{ to } j$.
    \\
    \indent \color{blue}{\textbf{Case 2b.1}: $d(1,u) = d(1,v) = d$.
    \\
    \indent Then if $e_i = (u,v)$ have:
    $$(N^Tx)_i = (-1)^{d(1,u)}\sqrt{w_{uv}}\frac{wt(\gamma_{1\rightarrow u})}{|wt(\gamma_{1\rightarrow u})|} +(-1)^{d(1,v)} \sqrt{\overline{w_{uv}}}\frac{wt(\gamma_{1\rightarrow v})}{|wt(\gamma_{1\rightarrow v})|}$$ 
    $$ = (-1)^d\sqrt{\overline{w_{uv}}} (\frac{w_{uv}}{|{w_{uv}}|}\frac{wt(\gamma_{1\rightarrow u})}{|wt(\gamma_{1\rightarrow u})|}+ \frac{wt(\gamma_{1\rightarrow v})}{|wt(\gamma_{1\rightarrow v})|})$$ 
    \\
    By Lemma \ref{dis}, let $\gamma'_{1\rightarrow v}$ be the shortest walk that together with $\gamma_{1\rightarrow u}$ and $\{u,v\}$ forms an odd unicyclic graph with odd cycle $C'$ where $e_k$ are all edges shared by $\gamma_{1\rightarrow u}$ and $\gamma'_{1\rightarrow v}$, and $w_{e_k}$ is the weight of $e_k$:
    $$wt(C')\prod_{k}|w_{e_k}|^2 = wt(\gamma_{1\rightarrow u})w_{uv}\overline{wt(\gamma'_{1\rightarrow v})}$$
    Since $wt(C') \in \R^-$. Then we have:
    $$ -1 = \frac{wt(C')}{|wt(C')|} = \frac{wt(\gamma_{1\rightarrow u})w_{uv}}{|wt(\gamma_{1\rightarrow u})w_{uv}|}\frac{\overline{wt(\gamma'_{1\rightarrow v})}}{|\overline{wt(\gamma'_{1\rightarrow v})}|}$$
    We can multiply both sides by $\frac{{wt(\gamma'_{1\rightarrow v})}}{|{wt(\gamma'_{1\rightarrow v})}|}$ to obtain:
    $$-\frac{{wt(\gamma'_{1\rightarrow v})}}{|{wt(\gamma'_{1\rightarrow v})}|} = \frac{wt(\gamma_{1\rightarrow u})w_{uv}}{|wt(\gamma_{1\rightarrow u})w_{uv}|}$$
    By Lemma \ref{shortestwalk}, we have
    $$
    -\frac{{wt(\gamma'_{1\rightarrow v})}}{|{wt(\gamma'_{1\rightarrow v})}|} = -\frac{{wt(\gamma_{1\rightarrow v})}}{|{wt(\gamma_{1\rightarrow v})}|} = \frac{wt(\gamma_{1\rightarrow u})w_{uv}}{|wt(\gamma_{1\rightarrow u})w_{uv}|}$$
    Thus, we conclude: 
    $$(N^Tx)_i = (-1)^d\sqrt{\overline{w_{uv}}} (\frac{w_{uv}}{|{w_{uv}}|}\frac{wt(\gamma_{1\rightarrow u})}{|wt(\gamma_{1\rightarrow u})|}+ \frac{wt(\gamma_{1\rightarrow v})}{|wt(\gamma_{1\rightarrow v})|}) = 0$$
    }
    \indent \color{blue}{\textbf{Case 2b.2}: $d(1,u) = d(1,v) \pm 1$.}
    \\
    \indent \indent \textbf{Case 2b.2.1}: $d(1,u) = d(1,v) - 1$.
    \\
    \indent \indent Then we have:
    $$(N^Tx)_i = (-1)^{d(1,u)}\sqrt{w_{uv}}\frac{wt(\gamma_{1\rightarrow u})}{|wt(\gamma_{1\rightarrow u})|} +(-1)^{d(1,v)} \sqrt{\overline{w_{uv}}}\frac{wt(\gamma_{1\rightarrow v})}{|wt(\gamma_{1\rightarrow v})|}$$ 
    $$ = (-1)^{d(1,u)}\sqrt{\overline{w_{uv}}} (\frac{w_{uv}}{|{w_{uv}}|}\frac{wt(\gamma_{1\rightarrow u})}{|wt(\gamma_{1\rightarrow u})|} - \frac{wt(\gamma_{1\rightarrow v})}{|wt(\gamma_{1\rightarrow v})|})$$
    By our assumption, $w_{uv}wt(\gamma_{1 \rightarrow u})$ is a shortest walk from $1$ to $v$. Thus by Lemma \ref{shortestwalk}, we have:
    $$(N^Tx)_i = (-1)^{d(1,u)}\sqrt{\overline{w_{uv}}} (\frac{w_{uv}}{|{w_{uv}}|}\frac{wt(\gamma_{1\rightarrow u})}{|wt(\gamma_{1\rightarrow u})|} - \frac{wt(\gamma_{1\rightarrow v})}{|wt(\gamma_{1\rightarrow v})|})= 0$$
    \indent \indent \textbf{Case 2b.2.2}: $d(1,u) = d(1,v) + 1$.
    \\
    \indent \indent Then we have:
    $$(N^Tx)_i = (-1)^{d(1,u)}\sqrt{w_{uv}}\frac{wt(\gamma_{1\rightarrow u})}{|wt(\gamma_{1\rightarrow u})|} +(-1)^{d(1,v)} \sqrt{\overline{w_{uv}}}\frac{wt(\gamma_{1\rightarrow v})}{|wt(\gamma_{1\rightarrow v})|}$$ 
    $$ = (-1)^{d(1,u)}\sqrt{{w_{uv}}}(\frac{wt(\gamma_{1\rightarrow u})}{|wt(\gamma_{1\rightarrow u})|} - \frac{wt(\gamma_{1\rightarrow v})}{|wt(\gamma_{1\rightarrow v})|}\frac{\overline{w_{uv}}}{|{\overline{w_{uv}}|}})$$
    By our assumption, $\overline{w_{uv}}wt(\gamma_{1 \rightarrow v})$ is a shortest walk from $1$ to $u$. Thus by Lemma \ref{shortestwalk}, we have:
    $$(N^Tx)_i = (-1)^{d(1,u)}\sqrt{{w_{uv}}}(\frac{wt(\gamma_{1\rightarrow u})}{|wt(\gamma_{1\rightarrow u})|} - \frac{wt(\gamma_{1\rightarrow v})}{|wt(\gamma_{1\rightarrow v})|}\frac{\overline{w_{uv}}}{|{\overline{w_{uv}}|}})= 0$$
    \\}
\end{proof}

For a weakly connected oriented graph $G$ on $n \geq 2$ vertices with complex edge weights, when $G$ is an oriented tree or all of its even oriented cycles have positive weights and all of its odd oriented cycles have negative weights, we can create a new $n \times n$ matrix $S = [s_{ij}]$ called the signed equiweight matrix of $G$, defined as follows:

\[
s_{ij} = (-1)^{d(i,j)} \frac{wt(\gamma_{i-j})}{|wt(\gamma_{i-j})|}, 
\]
with $d(i,j)$ denotes the shortest walk between vertex $i$ and vertex $j$.

This definition makes sense because of Theorem~\ref{thm:rank}. In addition, by Lemma~\ref{shortestwalk}, any walk $\gamma_{i-j}$ from a fixed vertex $i$ to another fixed vertex $j$ in $G$ has the same unit weight. Hence, our new signed equiweight matrix $S$ is well-defined. 

Note that $s_{ii} = (-1)^{d(i,i)}\frac{wt(\gamma_{i-i})}{|wt(\gamma_{i-i})|} = 1$.

\begin{theorem}
\label{thm:NN^+}
    Let $G$ be a weakly connected oriented graph on $n \geq 2$ vertices with complex edge weights and signless complex incidence matrix N. 

    (a) If $G$ has at least one even oriented cycle C such that $wt(C) \in \mathbb{C}\setminus\mathbb{R^+}$ or at least one odd oriented cycle C such that $wt(C) \in \mathbb{C}\setminus\mathbb{R^-}$, then 

    \[
    NN^+ = I_n.
    \]

    (b) If $G$ is an oriented tree or all of its even oriented cycles have positive weights and all of its odd oriented cycles have negative weights, then 

    \[NN^+ = I_n  - \frac{1}{n}S,\]

    where S is the signed equiweight matrix of G. 
\end{theorem}

\begin{proof}
    Let $N^+$ denote the Moore--Penrose inverse of the incidence matrix $N$. Given the property $NN^+N=N$ we have $(I_n - NN^+)N = I_nN-NN^+N=N-N=0.$ Then each row of $I_n - NN^+$ is in the left null space of $N$. 
    
    (a) Suppose G contains at least one even oriented cycle C such that $wt(C) \in \mathbb{C}\setminus\mathbb{R^+}$ or at least one odd oriented cycle C such that $wt(C) \in \mathbb{C}\setminus\mathbb{R^-}$. By Theorem~\ref{thm:rank}, $\rank(N) = n$ and thus the left null space of $N$ is trivial. Consequently, this implies $NN^+=I_n$. 

    (b) Suppose $G$ is an oriented tree or $G$ contains oriented cycles and $wt(C) \in \mathbb{R}^+$ for all even cycles and $wt(C) \in \mathbb{R}^-$ for all odd cycles. In addition, let $S$ be the signed equiweight matrix of $G$. By Theorem~\ref{thm:rank}, $\rank(N)=n-1$ and thus the left null space of $N$ has dimension 1. 

    Let $I_n - NN^+ =: [t_{ij}]$. Since 
    \[
    |t_{ij}| = |(-1)^{d(i,j)}c_i \frac{wt(\gamma_{i-j})}{|wt(\gamma_{i-j})|}| = |c_i| \frac{|wt(\gamma_{i-j})|}{|wt(\gamma_{i-j})|} = |c_i| 
    \]

    and $t_{ii} = (-1)^{d(i,i)}c_i =c_i$, the rest of the proof is similar to Theorem 2.5 in \cite{MP2026}.

\end{proof}

\begin{example}
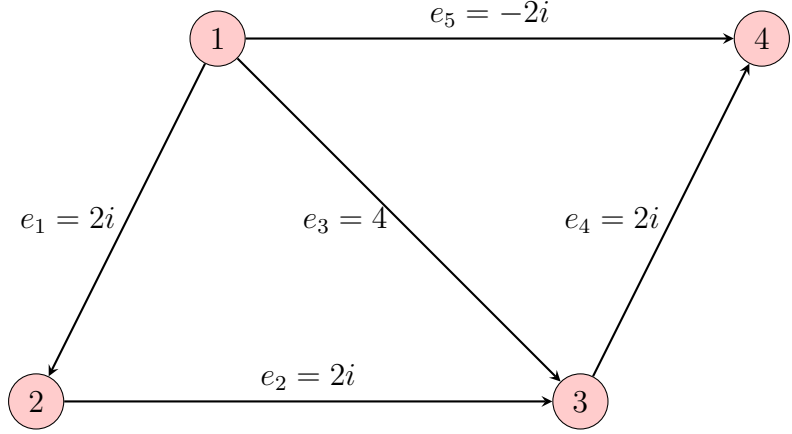
\begin{figure}[htbp]
\centering
\begin{tikzpicture}[scale=1.2,
                    colorstyle/.style={circle, fill, black, scale=.5},
                    >=stealth]
\tikzset{vertex/.style = {shape = circle, draw, minimum size = 1em, fill=pink!80}}
\tikzset{edge/.style = {->,thick}}

    \node[vertex] (1) at (0,4) {$1$};
    \node[vertex] (2) at (-2,0) {$2$};
    \node[vertex] (3) at (4,0) {$3$};
    \node[vertex] (4) at (6,4) {$4$};

    \draw[edge] (1) edge node[left]{$e_1 =2i$} (2);
    \draw[edge] (2) edge node[above]{$e_2 =2i$} (3);
    \draw[edge] (1) edge node[left]{$e_3 =4$} (3);
    \draw[edge] (1) edge node[above]{$e_5 =-2i$} (4);
    \draw[edge] (3) edge node[left]{$e_4 =2i$} (4);

    \end{tikzpicture}
    \caption{Example of a graph $G$ such that $wt(C) \in \R^+$ for even cycles and $wt(C) \in \R^-$ for odd cycles}.
    \end{figure}

    We know that the complex signless incidence matrix of $G$ is as follows: 

    \[
    N = 
    \begin{bmatrix}
        1+i & 0   & 2 & 0   & 1-i \\
        1-i & 1+i & 0 & 0   & 0   \\
        0   & 1-i & 2 & 1+i & 0   \\
        0   & 0   & 0 & 1-i & 1+i
    \end{bmatrix}
\]

    First, we show that the example above satisfies the condition in case (b). Let's denote the odd cycle on the left, the odd cycle on the right, and the even cycle  $C_{H_1}$, $C_{H_2}$, and $C_{H_3}$ in that order. 

    \vspace{0.5em}
    
    We have: 

    \begin{align*}
wt(C_{H_1})
    &= w_{12}w_{23}\overline{w_{13}}
    &= (2i)(2i)(4)
    &= -16,\\
wt(C_{H_2})
    &= w_{13}w_{34}\overline{w_{14}}
    &= (4)(2i)(2i)
    &= -16,\\
wt(C_{H_3})
    &= w_{12}w_{23}w_{34}\overline{w_{14}}
    &= (2i)(2i)(2i)(2i)
    &= 16.
\end{align*}

Hence, the example graph $G$ satisfies the condition in case (b) of Theorem~\ref{thm:NN^+}. Therefore, $NN+ = I_n  - \frac{1}{n}S$, where $N$ is the complex signless incidence matrix of $G$ and $S$ is the signed equiweight matrix. 

According to SageMath,

\[
NN^+= I_n-\frac{1}{n}S=
\left[
\begin{array}{rrrr}
 \frac{3}{4} &  \frac{1}{4}i &  \frac{1}{4} & -\frac{1}{4}i \\
-\frac{1}{4}i &  \frac{3}{4} &  \frac{1}{4}i &  \frac{1}{4} \\
 \frac{1}{4} & -\frac{1}{4}i &  \frac{3}{4} &  \frac{1}{4}i \\
 \frac{1}{4}i &  \frac{1}{4} & -\frac{1}{4}i &  \frac{3}{4}
\end{array}
\right]
\]

\end{example}

\section{Combinatorial Formula of \texorpdfstring{$N^+$}{N+} for Oriented Trees and Certain Unicyclic Graphs}

\begin{example}
    \begin{figure}[htbp]
    \centering
    \begin{tikzpicture}[
    node distance=2cm and 3cm,
    vertex/.style={
        circle, 
        draw=black, 
        fill=pink!80, 
        minimum size=5mm, 
        inner sep=0pt,
        font=\small\bfseries
    },
    edge/.style={
        ->, 
        >={Stealth[scale=2]}, 
        thick, 
        shorten >=1pt
    },
    weight/.style={
        font=\small\color{blue}, 
        fill=white, 
        inner sep=2pt
    }
    ]

    \node[vertex] (4) at (8,0) {4};
    \node[vertex] (2) at (4,0) {2};
    \node[vertex] (1) at (3,-2) {1};
    \node[vertex] (3) at (3,2) {3};
        
    \draw[->] (2) -- node[weight] {4 $: e_3$} (4);
    \draw[->] (1) -- node[weight] {9 $: e_1$} (2);
    \draw[<-] (3) -- node[weight] {i $: e_2$} (2);

    \end{tikzpicture}
    \caption{\color{blue}{Example of a graph $G$ such that $G$ is an oriented tree}}
    \end{figure}
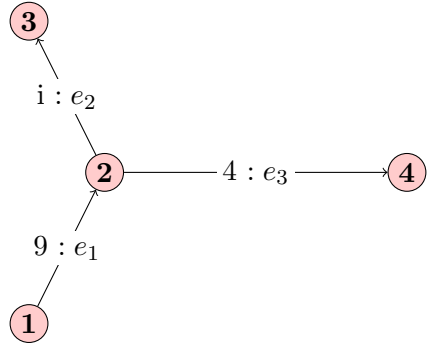
    
    \color{blue}{
    The signless incidence matrix for $G$ is: $$N =
    \begin{bmatrix}
    3 & 0 & 0 \\
    3 & \sqrt{\i} & 2 \\
    0 & \sqrt{-\i} & 0 \\
    0 & 0 & 2
    \end{bmatrix}
    $$
    Using Sage Math, one can confirm that the Moore--Penrose inverse of $N$ is: $$N^+ = \frac{1}{4}
    \begin{bmatrix}
     1 & \frac{1}{3} & \frac{1}{3\i} & -\frac{1}{3} \\
    -\frac{1}{\sqrt{\i}} & \frac{1}{\sqrt{\i}} & \frac{3}{\sqrt{-\i}} & -\frac{1}{\sqrt{\i}} \\
    -\frac{1}{2} & \frac{1}{2} & \frac{1}{2\i} & \frac{3}{2}
    \end{bmatrix}
    $$
    \\
    According to Thm 3.2, we have:
    $$(N^+)_{11} = {\frac{1}{4}(-1)^{d(e_1,1)}\frac{3}{\sqrt{9}}\frac{wt(\gamma_{1 \rightarrow 1})}{|wt(\gamma_{1 \rightarrow 1})|}}= \frac{1}{4} $$

    $$(N^+)_{34} = {\frac{1}{4}(-1)^{d(e_3,4)}\frac{3}{\sqrt{4}}\frac{wt(\gamma_{4 \rightarrow 4})}{|wt(\gamma_{4 \rightarrow 4})|}}= \frac{3}{8} $$

    $$(N^+)_{23} ={\frac{1}{4}(-1)^{d(e_2,3)}\frac{3}{\sqrt{-\i}}\frac{wt(\gamma_{3 \rightarrow 3})}{|wt(\gamma_{3 \rightarrow 3})|}}= \frac{3}{4\sqrt{-\i}} $$
    }
\end{example}

\textcolor{blue}{Let $T$ be an oriented tree on $n \geq 2$ vertices with directed edges $e_1,e_2,...,e_{n-1}$. Suppose $e_i = (u,v)$ is an edge in $T$. Let $T - e_i$ denote the subdigraph obtained from $T$ by removing $e_i$. Then, following the notation used in \cite{MP2026}, we let $T_t(e_i)$ (resp. $T_h(e_i)$) denote the connected component of $T - e_i$ containing the tail vertex of $e_i$, $u$, (resp. the head vertex of $e_i$, $v$). In addition, $|T_t(e_i)|$ (resp. $|T_h(e_i)|$) denotes the number of vertices in the component of $T - e_i$ containing the tail vertex $u$ (resp. the head vertex $v$). Also, using the notation established in \cite{HM1}, we define $d(e_i, j) = \min(d(u,j), d(v,j))$.}

\begin{theorem}\label{mptree}
    \textcolor{blue}{Let $T$ be an oriented tree on $n \geq 2$ vertices with directed edges $e_1,e_2,...,e_{n-1}$, non-zero edge weights in $\C \setminus \R^-$, and complex signless incidence matrix $N$. Define $B$ as an $(n-1) \times n$ matrix such that 
    $$b_{ij} = \frac{(-1)^{d(e_i,j)}}{n}
    \begin{cases}
        \frac{|T_h(e_i)|}{\sqrt{w_{uv}}}\frac{|wt(\gamma_{j \rightarrow u})|}{wt(\gamma_{j \rightarrow u})} & \text{where } j \in T_t(e_i)
        \\
        \\
        \frac{|T_t(e_i)|}{\sqrt{\overline{w_{uv}}}}\frac{|wt(\gamma_{j \rightarrow v})|}{wt(\gamma_{j \rightarrow v})} & \text{where } j \in T_h(e_i) 
    \end{cases}
    $$
    where $e_i = (u,v)$. Then, $B$ is the Moore--Penrose inverse of $N$.
    }
\end{theorem}

\begin{proof}
    \textcolor{blue}{We will first show $BN = I_{n-1}$.} \\
    
    \textcolor{blue}{\textbf{Case 1: $i = j$}. 
    \\
    \indent Let $e_i = (u,v)$.}
    
    \color{blue}{$$(BN)_{ii} = b_{iu}\sqrt{w_{uv}} + b_{iv}\sqrt{\overline{w_{uv}}}$$  
        
    $$= \frac{(-1)^{d(e_i,u)}}{n}\frac{|T_h(e_i)|}{\sqrt{w_{uv}}}\frac{|wt(\gamma_{u \rightarrow u})|}{wt(\gamma_{u \rightarrow u})}\sqrt{w_{uv}}+ \frac{(-1)^{d(e_i,v)}}{n}\frac{|T_t(e_i)|}{\sqrt{\overline{w_{uv}}}}\frac{|wt(\gamma_{v \rightarrow v})|}{wt(\gamma_{v \rightarrow v})}\sqrt{\overline{w_{uv}}}$$
        
    $$= \frac{(-1)^{d(u,u)}}{n}\frac{|T_h(e_i)|}{\sqrt{w_{uv}}}\frac{|wt(\gamma_{u \rightarrow u})|}{wt(\gamma_{u \rightarrow u})}\sqrt{w_{uv}}+ \frac{(-1)^{d(v,v)}}{n}\frac{|T_t(e_i)|}{\sqrt{\overline{w_{uv}}}}\frac{|wt(\gamma_{v \rightarrow v})|}{wt(\gamma_{v \rightarrow v})}\sqrt{\overline{w_{uv}}}$$
    $$ = \frac{|T_h(e_i)|}{n} + \frac{|T_t(e_i)|}{n} = 1$$}

    \textcolor{blue}{\textbf{Case 2: $i \neq j$}. 
    \\
    \indent Let $e_i = (u,v)$ and $e_j = (k,\ell)$.} \\
    \indent \textcolor{blue}{\textbf{Case 2a}:} $k,\ell \in T_t(e_i)$. 
    \\
    \indent Without loss of generality, assume $k$ lies in $\gamma_{\ell \rightarrow u}$.
    
    \color{blue}{$$(BN)_{ij} = b_{ik}\sqrt{w_{k\ell}} + b_{i\ell}\sqrt{\overline{w_{k\ell}}}$$

    $$= \frac{(-1)^{d(e_i,k)}}{n}\frac{|T_h(e_i)|}{\sqrt{w_{uv}}}\frac{|wt(\gamma_{k \rightarrow u})|}{wt(\gamma_{k \rightarrow u})}\sqrt{w_{kl}}+ \frac{(-1)^{d(e_i,\ell)}}{n}\frac{|T_h(e_i)|}{\sqrt{{w_{uv}}}}\frac{|wt(\gamma_{\ell \rightarrow u})|}{wt(\gamma_{\ell \rightarrow u})}\sqrt{\overline{w_{k\ell}}}$$

    $$= \frac{(-1)^{d(e_i,k)}}{n}\sqrt{\overline{w_{k\ell}}} \left(\frac{|T_h(e_i)|}{\sqrt{w_{uv}}}\frac{|wt(\gamma_{k \rightarrow u})|}{wt(\gamma_{k \rightarrow u})}\frac{|\overline{w_{k\ell}}|}{\overline{w_{k\ell}}} - \frac{|T_h(e_i)|}{\sqrt{{w_{uv}}}}\frac{|wt(\gamma_{\ell \rightarrow u})|}{wt(\gamma_{\ell \rightarrow u})}\right)$$}
    \indent By assumption, we have:
    
    $$\frac{|wt(\gamma_{\ell \rightarrow u})|}{wt(\gamma_{\ell \rightarrow u})} = \frac{|wt(\gamma_{k \rightarrow u})|}{wt(\gamma_{k \rightarrow u})}\frac{|\overline{w_{k\ell}}|}{\overline{w_{k\ell}}}$$
    \indent Plugging this into the equation above, we see:
    $$(BN)_{ij} = 0$$
    \indent \textbf{Case 2b:} $k,\ell \in T_h(e_i)$. 
    
    \indent Without loss of generality, assume k lies in $\gamma_{\ell \rightarrow u}$. 
    
    \indent The proof is identical to that of Case 2a.

    Since $BN = I_{n-1}$, we get $BNB = B$, $NBN = N$, and $(BN)^* = BN$. Therefore, we only have to show $(NB)^* = NB$. 
    Let the edges incident at $i$ be $e_k = (i,k)$ for all $k = 1,2,...,r$ and $e_k = (k,i)$ for all $k = r+1,r+2,...,s$. 
    
\textbf{Case 1:} $i \neq j$.

Without loss of generality, let $e_1$ lie on the walk from vertex i to vertex j. Note that $d(e_k,j) = d(i,j)$ since $d(k,j) = d(i,j) + 1, k \neq 1$.
    $$(NB)_{ij} = (-1)^{d(e_1,j)}\left(\frac{|T_t(e_1)|}{n {\sqrt{\overline{w_{i1}}}}}\frac{|wt(\gamma_{j \rightarrow 1})|}{wt(\gamma_{j \rightarrow 1})}\right)\sqrt{w_{i1}} + \sum_{k=2}^r(-1)^{d(e_k,j)} \left(\frac{|T_h(e_k)|}{n {\sqrt{w_{ik}}}}\frac{|wt(\gamma_{j \rightarrow i})|}{wt(\gamma_{j \rightarrow i})}\right)\sqrt{w_{ik}} 
    $$
    $$+ \sum_{k=r+1}^s(-1)^{d(e_k,j)}\left(\frac{|T_t(e_k)|}{n {\sqrt{\overline{w_{ki}}}}}\frac{|wt(\gamma_{j \rightarrow i})|}{wt(\gamma_{j \rightarrow i})} \right)\sqrt{\overline{w_{ki}}}$$
    $$= (-1)^{d(e_1,j)}\left(\frac{|T_t(e_1)|}{n {\sqrt{\overline{w_{i1}}}}}\frac{|wt(\gamma_{j \rightarrow 1})|}{wt(\gamma_{j \rightarrow 1})}\right)\sqrt{w_{i1}} + (-1)^{d(j,i)}\frac{|wt(\gamma_{j \rightarrow i})|}{nwt(\gamma_{j \rightarrow i})}\left(\sum_{k=2}^r|T_h(e_k)| + \sum_{k=r+1}^s|T_t(e_k)|\right)$$
    
    $$= (-1)^{d(1,j)}\left(\frac{|T_t(e_1)|}{n {\sqrt{\overline{w_{i1}}}}}\frac{|wt(\gamma_{j \rightarrow 1})|}{wt(\gamma_{j \rightarrow 1})}\right)\sqrt{w_{i1}} + (-1)^{d(j,i)}\frac{|wt(\gamma_{j \rightarrow i})|}{nwt(\gamma_{j \rightarrow i})}\left(|T_t(e_1)|-1\right) $$
    By our assumption that $e_1$ lies on the walk from vertex i to vertex j, we have $(-1)^{d(j,1)} = -(-1)^{d(j,i)}$. So, we get:
    $$(NB)_{ij} = -(-1)^{d(j,i)}\frac{|wt(\gamma_{j \rightarrow i})|}{nwt(\gamma_{j \rightarrow i})}$$
    Then we deduce:
    $$\overline{(NB)_{ji}} = -(-1)^{d(i,j)}\frac{|\overline{wt({\gamma_{i \rightarrow j}})|}}{n\overline{wt(\gamma_{i \rightarrow j})}} =-(-1)^{d(j,i)}\frac{|wt(\gamma_{j \rightarrow i})|}{nwt(\gamma_{j \rightarrow i})} =(NB)_{ij}$$

\textbf{Case 2:} $i = j$.     
    $$(NB)_{ii} = (-1)^{d(e_1,i)}\left(\frac{|T_h(e_1)|}{n {\sqrt{{w_{i1}}}}}\frac{|wt(\gamma_{i \rightarrow i})|}{wt(\gamma_{i \rightarrow i})}\right)\sqrt{w_{i1}} + \sum_{k=2}^r(-1)^{d(e_k,i)} \left(\frac{|T_h(e_k)|}{n {\sqrt{w_{ik}}}}\frac{|wt(\gamma_{i \rightarrow i})|}{wt(\gamma_{i \rightarrow i})}\right)\sqrt{w_{ik}} 
    $$
    $$+ \sum_{k=r+1}^s(-1)^{d(e_k,i)}\left(\frac{|T_t(e_k)|}{n {\sqrt{\overline{w_{ki}}}}}\frac{|wt(\gamma_{i \rightarrow i})|}{wt(\gamma_{i \rightarrow i})} \right)\sqrt{\overline{w_{ki}}}$$

    $$= |T_h(e_1)| + \sum_{k=2}^r\frac{|T_h(e_k)|}{n} + \sum_{k=r+1}^s\frac{|T_t(e_k)|}{n} \in \R$$
    So, $(NB)_{ii} = \overline{(NB)_{ii}}$.
    Thus, we have shown that $B = N^+$.

\end{proof}

If $G$ is an oriented unicyclic graph on $n \geq 2$ vertices with cycle $C$, and $e_i$ is not in the edges that form $C$, then removing $e_i$ leaves two connected components. 

One of the components will be an oriented unicyclic graph, denoted as $G \setminus e_i[C]$. The other component will be an oriented tree which we will call $G \setminus e_i(C)$. 

If $e_i = (u,v)$, then $G \setminus (e_i,v](C)$ says that the tree component contains vertex $v$ while $G \setminus (e_i,u](C)$ says that the tree component contains vertex $u$.

\begin{example}
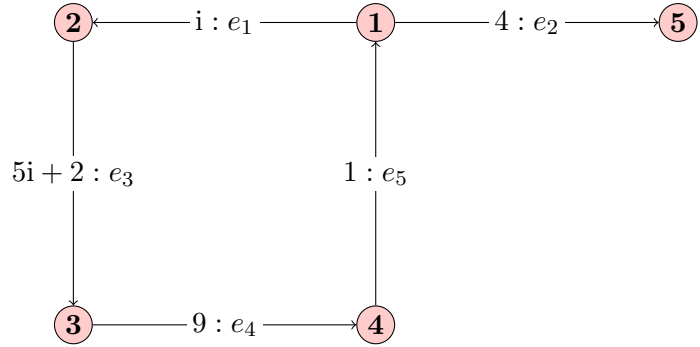
\begin{figure}[htbp]
\centering
\begin{tikzpicture}[
    node distance=2cm and 3cm,
    vertex/.style={
        circle, 
        draw=black, 
        fill=pink!80, 
        minimum size=5mm, 
        inner sep=0pt,
        font=\small\bfseries
    },
    edge/.style={
        ->, 
        >={Stealth[scale=2]}, 
        thick, 
        shorten >=1pt
    },
    weight/.style={
        font=\small\color{blue}, 
        fill=white, 
        inner sep=2pt
    }
    ]

    \node[vertex] (1) at (8,0) {1};
    \node[vertex] (2) at (8,-4) {4};
    
    \node[vertex] (4) at (4,0) {2};
    \node[vertex] (3) at (4,-4) {3};
    
    \node[vertex] (5) at (12,0) {5};
    
    \draw[->] (1) -- node[weight] {$4 : e_2$} (5);
    \draw[->] (1) -- node[weight] {$\mathrm{i} : e_1$} (4);
    \draw[->] (4) -- node[weight] {$5\mathrm{i}+2 : e_3$} (3);
    \draw[->] (3) -- node[weight] {$9 : e_4$} (2);
    \draw[->] (2) -- node[weight] {$1 : e_5$} (1);

    \end{tikzpicture}
    \caption{\color{blue}{Example of a graph $G$ such that $G$ is even unicyclic graph with $wt(C) \notin \R^+$}.}
    \end{figure}
    
    \color{blue}{
    According to Sage Math and rounding to the fifth digit if needed, $(N^+)_{11} = 0.28547 - 0.42164\i$, $(N^+)_{23} = 0$, $(N^+)_{25} = 0.5$, and $(N^+)_{51} = 0.5 + 0.09629\i$.
    \\
    According to Theorem \ref{mpeven}, we have:
    $$(N^+)_{11} = \frac{(-1)^{d(1,1)}\frac{wt(\gamma_{1 \rightarrow 1})}{|wt(\gamma_{1 \rightarrow 1})|}}{{\sqrt{-\i}}(\i - \frac{2-5\i}{\sqrt{29}})} = 0.28547 - 0.42164\i $$

    $$(N^+)_{23} = 0 = 0  $$

    $$(N^+)_{25} = {(-1)^{d(5,5)}\frac{1}{\sqrt4}\frac{wt(\gamma_{5 \rightarrow 5})}{|wt(\gamma_{5 \rightarrow 5})|}}= \frac{1}{2} = 0.5 $$

    $$(N^+)_{51} = \frac{(-1)^{d(1,4)}\frac{wt(\gamma_{4 \rightarrow 1})}{|wt(\gamma_{4 \rightarrow 1})|}}{{1}(1 - \frac{wt(\gamma_{4 \rightarrow 1})}{|wt(\gamma_{4 \rightarrow 1})|})} =  \frac{(-1)\frac{-\i(2-5\i)}{\sqrt{29}}}{{1}(1 + \frac{\i(2-5\i)}{\sqrt{29}})} = 0.5 + 0.09629\i$$
    
    }    
\end{example}

\begin{theorem}\label{mpeven}
    \color{blue}{Let $G$ be an oriented unicyclic graph on $n \geq 2$ vertices with even cycle $C$, non-zero edge weights in $\C \setminus \R^-$,  and complex signless incidence matrix $N$. If $wt(C) \in \C \setminus \R^+$, then by Lemma \ref{barik}, we can define the $n \times n$ matrix $F$ such that 
    $$f_{ij} = 
    \begin{cases}
        (-1)^{d^*(j,u)}\frac{\frac{wt(\gamma_{u \rightarrow j})}{|wt(\gamma_{u \rightarrow j})|}}{\sqrt{\overline{w_{uv}}}(\frac{w_{uv}}{|w_{uv}|}-\frac{wt(\gamma_{u \rightarrow v})}{|wt(\gamma_{u \rightarrow v})|})} & \text{where } e_i = (u,v) \in C
        \\
        \\
        0 & \text{where } e_i \notin C \text{and } j \in G \setminus e_i[C] 
        \\
        \\
        (-1)^{d^*(j,u)}\frac{1}{\sqrt{{w_{uv}}}}\frac{|wt(\gamma_{j \rightarrow u})|}{wt(\gamma_{j \rightarrow u})} & \text{where } e_i \notin C \text{and } j \in G \setminus (e_i,u](C) 
        \\
        \\
        (-1)^{d^*(j,v)}\frac{1}{\sqrt{\overline{w_{uv}}}}\frac{|wt(\gamma_{j \rightarrow v})|}{wt(\gamma_{j \rightarrow v})} & \text{where } e_i \notin C \text{and } j \in G \setminus (e_i,v](C) 
    \end{cases}
    $$
     where $d^*(s,r)$ and $\gamma_{s \rightarrow r}$ denote the 
    distance and walk, respectively, between vertex $s$ and vertex $r$ in the underlying tree that is formed by removing $e_i = (u,v)$ from $G$. Then, $F$ is the Moore--Penrose inverse of $N$.  
    }
\end{theorem}

\begin{proof}
    \color{blue}{
    It suffices to show $FN = I_n$ as Theorem \ref{thm:rank} implies that $N$ is invertible and thus $F = N^{-1}=N^+$. 
    \\
    \textbf{Case 1}: $i = j$.
    \\
    \textbf{Case 1a}: $e_i = (u,v) \in C$.
    $$(FN)_{ii} = f_{iu}\sqrt{w_{uv}} + f_{iv}\sqrt{\overline{w_{uv}}}$$
    $$= (-1)^{d^*(u,u)}\frac{\frac{wt(\gamma_{u \rightarrow u})}{|wt(\gamma_{u \rightarrow u})|}}{\sqrt{\overline{w_{uv}}}(\frac{w_{uv}}{|w_{uv}|}-\frac{wt(\gamma_{u \rightarrow v})}{|wt(\gamma_{u \rightarrow v})|})} \sqrt{w_{uv}} + (-1)^{d^*(u,v)}\frac{\frac{wt(\gamma_{u \rightarrow v})}{|wt(\gamma_{u \rightarrow v})|}}{\sqrt{\overline{w_{uv}}}(\frac{w_{uv}}{|w_{uv}|}-\frac{wt(\gamma_{u \rightarrow v})}{|wt(\gamma_{u \rightarrow v})|})} \sqrt{\overline{w_{uv}}}$$
    because $C$ is an even cycle, $(-1)^{d^*(u,v)} = -1$. Thus, we get:
    $$= \frac{\frac{w_{uv}}{|w_{uv}|}}{\frac{w_{uv}}{|w_{uv}|}-\frac{wt(\gamma_{u \rightarrow v})}{|wt(\gamma_{u \rightarrow v})|}} - \frac{\frac{wt(\gamma_{u \rightarrow v})}{|wt(\gamma_{u \rightarrow v})|}}{\frac{w_{uv}}{|w_{uv}|}-\frac{wt(\gamma_{u \rightarrow v})}{|wt(\gamma_{u \rightarrow v})|}} = 1$$
    \\
    \textbf{Case 1b}: $e_i = (u,v) \notin C$.
    \\
     Without loss of generality, let $u \in G \setminus e_i(C)$ and $v \in G \setminus e_i[C]$.
    $$(FN)_{ii} = f_{iu}\sqrt{w_{uv}} + f_{iv}\sqrt{\overline{w_{uv}}}$$
    $$= (-1)^{d^*(u,u)}\frac{1}{\sqrt{{w_{uv}}}}\frac{|wt(\gamma_{u \rightarrow u})|}{wt(\gamma_{u \rightarrow u})} \sqrt{w_{uv}}- 0\sqrt{\overline{w_{uv}}} =1 $$
    \\
    \textbf{Case 2}: $i \neq j$
    \\
     Let $e_i = (u,v)$ and $e_j = (k,\ell)$.
    \\
    \textbf{Case 2a}: $e_i = (u,v) \in C$. 
    \\
    $$(FN)_{ij} = f_{ik}\sqrt{w_{k\ell}} + f_{i\ell}\sqrt{\overline{w_{k\ell}}}$$
    $$= (-1)^{d^*(k,u)}\frac{\frac{wt(\gamma_{u \rightarrow k})}{|wt(\gamma_{u \rightarrow k})|}}{\sqrt{\overline{w_{uv}}}(\frac{w_{uv}}{|w_{uv}|}+\frac{wt(\gamma_{u \rightarrow v})}{|wt(\gamma_{u \rightarrow v})|})} \sqrt{w_{k\ell}} + (-1)^{d^*(\ell,u)}\frac{\frac{wt(\gamma_{u \rightarrow \ell})}{|wt(\gamma_{u \rightarrow \ell})|}}{\sqrt{\overline{w_{uv}}}(\frac{w_{uv}}{|w_{uv}|}+\frac{wt(\gamma_{u \rightarrow v})}{|wt(\gamma_{u \rightarrow v})|})} \sqrt{\overline{w_{k\ell}}}$$
    $$= (-1)^{d^*(k,u)}\sqrt{\overline{w_{k\ell}}}\frac{1}{\sqrt{\overline{w_{uv}}}(\frac{w_{uv}}{|w_{uv}|}+\frac{wt(\gamma_{u \rightarrow v})}{|wt(\gamma_{u \rightarrow v})|})} \left( \frac{wt(\gamma_{u \rightarrow k})}{|wt(\gamma_{u \rightarrow k})|}\frac{w_{k\ell}}{|w_{k\ell}|}-\frac{wt(\gamma_{u \rightarrow \ell})}{|wt(\gamma_{u \rightarrow \ell})|} \right) = 0$$
    where the third equality holds as $k$ either lies in $\gamma_{u \rightarrow \ell}$ or $\ell$ lies in $\gamma_{u \rightarrow k}$ which implies that $d^*(\ell,u) = d^*(k,u) \pm 1$, and Observation \ref{treeunitweight} implies that $\frac{wt(\gamma_{u \rightarrow k})}{|wt(\gamma_{u \rightarrow k})|}\frac{w_{k\ell}}{|w_{k\ell}|} = \frac{wt(\gamma_{u \rightarrow \ell})}{|wt(\gamma_{u \rightarrow \ell})|}$.
        
    \textbf{Case 2b}: $e_i = (u,v) \notin C$. 
    \\
    \textbf{Case 2b.1}: $k,\ell \in G\setminus e_i[C]$. 
    $$(FN)_{ij} = f_{ik}\sqrt{w_{k\ell}} + f_{il}\sqrt{\overline{w_{k\ell}}} = 0 + 0 = 0$$
    \\
    \textbf{Case 2b.2}: $k,\ell \in G\setminus (e_i, v](C)$ or $k,\ell \in G\setminus (e_i, u](C)$.
    \\
    Without loss of generality, we can assume $k,\ell \in G\setminus (e_i, v](C)$.
    
    $$(FN)_{ij} = f_{ik}\sqrt{w_{k\ell}} + f_{il}\sqrt{\overline{w_{k\ell}}}$$
    $$= (-1)^{d^*(k,v)}\frac{1}{\sqrt{\overline{w_{uv}}}}\frac{|wt(\gamma_{k \rightarrow v})|}{wt(\gamma_{k \rightarrow v})}\sqrt{w_{k\ell}} + (-1)^{d^*(\ell,v)}\frac{1}{\sqrt{\overline{w_{uv}}}}\frac{|wt(\gamma_{\ell \rightarrow v})|}{wt(\gamma_{\ell \rightarrow v})}\sqrt{\overline{w_{k\ell}}}$$
    $$= (-1)^{d^*(k,v)}\frac{1}{\sqrt{\overline{w_{uv}}}}\sqrt{w_{k\ell}}\left(\frac{|wt(\gamma_{k \rightarrow v})|}{wt(\gamma_{k \rightarrow v})} - \frac{|wt(\gamma_{\ell \rightarrow v})|}{wt(\gamma_{\ell \rightarrow v})}\frac{|w_{k\ell}|}{w_{k\ell}} \right) = 0$$
     where the third equality holds as $k$ either lies in $\gamma_{\ell \rightarrow v}$ or $\ell$ lies in $\gamma_{k \rightarrow v}$ which implies that $d^*(\ell,u) = d^*(k,u) \pm 1$, and Observation \ref{treeunitweight} implies that $\frac{|wt(\gamma_{\ell \rightarrow v})|}{wt(\gamma_{\ell \rightarrow v})}\frac{|w_{k\ell}|}{w_{k\ell}} = \frac{|wt(\gamma_{k \rightarrow v})|}{wt(\gamma_{k \rightarrow v})}$.
    }
\end{proof}

\begin{example}
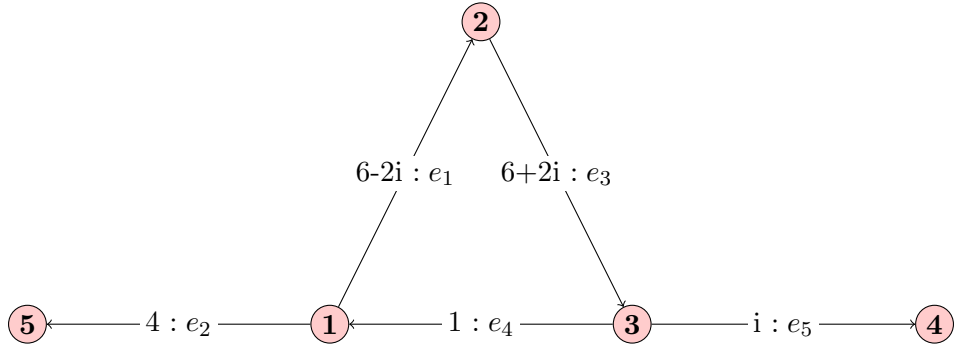
\begin{figure}[htbp]
\centering
\begin{tikzpicture}[
    node distance=2cm and 3cm,
    vertex/.style={
        circle, 
        draw=black, 
        fill=pink!80, 
        minimum size=5mm, 
        inner sep=0pt,
        font=\small\bfseries
    },
    edge/.style={
        ->, 
        >={Stealth[scale=2]}, 
        thick, 
        shorten >=1pt
    },
    weight/.style={
        font=\small\color{blue}, 
        fill=white, 
        inner sep=2pt
    }
    ]

    \node[vertex] (3) at (8,0) {3};
    \node[vertex] (2) at (6,4) {2};
    \node[vertex] (1) at (4,0) {1};
    \node[vertex] (4) at (12,0) {4};
    \node[vertex] (5) at (0,0) {5};

    \draw[<-] (2) -- node[weight] {6-2i $: e_1$} (1);
    \draw[<-] (1) -- node[weight] {1 $: e_4$} (3);
    \draw[<-] (4) -- node[weight] {i $: e_5$} (3);
    \draw[<-] (3) -- node[weight] {6+2i $: e_3$} (2);
    \draw[<-] (5) -- node[weight] {4 $: e_2$} (1);

    \end{tikzpicture}
    \caption{\color{blue}{Example of a graph $G$ such that $G$ is an odd unicyclic graph with $wt(C) \in \R^+$}.}
    \end{figure}
    
    \color{blue}{
    According to Sage Math and rounding to the fifth digit if needed, $(N^+)_{54} = 0.70711 + 0.70711\i$, $(N^+)_{33} = 0.19625 + 0.03185\i$, $(N^+)_{14} = -0.03185 + 0.19625\i$, and $(N^+)_{44} = -0.5\i$. 
    \\
    According to Theorem \ref{mpodd}, we have:
    $$(N^+)_{54} = {(-1)^{d^*(4,4)}\frac{1}{\sqrt-\i}\frac{wt(\gamma_{4 \rightarrow 4})}{|wt(\gamma_{4 \rightarrow 4})|}}= \frac{1}{\sqrt-\i} = 0.70711 + 0.70711\i  $$

    $$(N^+)_{33} = \frac{(-1)^{d^*(3,2)}\frac{wt(\gamma_{2 \rightarrow 3})}{|wt(\gamma_{2 \rightarrow 3})|}}{{\sqrt{6-2\i}}(\frac{6+2\i}{\sqrt40} + \frac{6+2\i}{\sqrt40})} = \frac{6+2\i}{\sqrt{40}}\frac{1}{{\sqrt{6-2\i}}(\frac{6+2\i}{\sqrt{40}} + \frac{6+2\i}{\sqrt{40}})} = \frac{1}{2\sqrt{6-2\i}} = 0.19625 + 0.03185\i  $$

    $$(N^+)_{14} = \frac{(-1)^{d^*(4,1)}\frac{wt(\gamma_{1 \rightarrow 4})}{|wt(\gamma_{1 \rightarrow 4})|}}{{\sqrt{6+2\i}}(2\frac{6-2\i}{\sqrt{40}})} = i\frac{1}{{\sqrt{6+2\i}}(2\frac{6-2\i}{\sqrt{40}})} = -0.03185 + 0.19625\i $$

    $$(N^+)_{44} =\frac{(-1)^{d^*(4,3)}\frac{wt(\gamma_{3 \rightarrow 4})}{|wt(\gamma_{3 \rightarrow 4})|}}{{\sqrt{1}}(1 + 1)} = -0.5\i  $$

    }        
\end{example}

\begin{lemma}\label{oddwelldef}
    \textcolor{blue}{Let $G$ be an oriented unicyclic graph on $n \geq 2$ vertices with cycle $C$, vertices $i$ and $j$, and non-zero edge weights in $\C \setminus \R^-$. Also, let $\gamma'_{i \rightarrow j}$ and $\gamma_{i \rightarrow j}$ be any two walks from vertex $i$ to vertex $j$.
    If $wt(C) \in \C \setminus \R^-$, then $\frac{wt(\gamma'_{i \rightarrow j})}{|wt(\gamma'_{i \rightarrow j})|} \neq - \frac{wt(\gamma_{i \rightarrow j})}{|wt(\gamma_{i \rightarrow j})|}$}.
\end{lemma}

\begin{proof}

    \textbf{Case 1:} $wt(C) \in \R^+$. 
    
    In this scenario, Lemma \ref{barik} implies that $\frac{wt(\gamma'_{i \rightarrow j})}{|wt(\gamma'_{i \rightarrow j})|} = \frac{wt(\gamma_{i \rightarrow j})}{|wt(\gamma_{i \rightarrow j})|}$. 
    
    \textbf{Case 2:} $wt(C) \in \C \setminus \R$. 
    $$\frac{wt(\gamma'_{i \rightarrow j})}{wt(\gamma_{i \rightarrow j})} = \frac{wt(\gamma'_{i \rightarrow j}[C])}{wt(\gamma_{i \rightarrow j}[C])} = \frac{wt(\gamma'_{i \rightarrow j}[C])}{wt(\gamma_{i \rightarrow j}[C])}\frac{\overline{wt(\gamma_{i \rightarrow j}[C])}}{\overline{wt(\gamma_{i \rightarrow j}[C])}} = \frac{wt(C)}{|wt(\gamma_{i \rightarrow j})|^2} \in \C \setminus \R$$
    Thus, $\frac{wt(\gamma'_{i \rightarrow j})}{|wt(\gamma'_{i \rightarrow j})|} \neq - \frac{wt(\gamma_{i \rightarrow j})}{|wt(\gamma_{i \rightarrow j})|}$ for if not, then we have $\frac{wt(\gamma'_{i \rightarrow j})}{wt(\gamma_{i \rightarrow j})} \in \R^-$.
\end{proof}

\begin{theorem}\label{mpodd}
    \color{blue}{Let $G$ be an oriented unicyclic graph on $n \geq 2$ vertices with odd cycle $C$, non-zero edge weights in $\C \setminus \R^-$, and complex signless incidence matrix $N$. If $wt(C) \in \C \setminus \R^-$, then by Lemma \ref{oddwelldef}, we can define the $n \times n$ matrix $H$ such that
    $$h_{ij} = 
    \begin{cases}
        (-1)^{d^*(j,u)}\frac{\frac{wt(\gamma_{u \rightarrow j})}{|wt(\gamma_{u \rightarrow j})|}}{\sqrt{\overline{w_{uv}}}(\frac{w_{uv}}{|w_{uv}|}+\frac{wt(\gamma_{u \rightarrow v})}{|wt(\gamma_{u \rightarrow v})|})} & \text{where } e_i = (u,v) \in C
        \\
        \\
        0 & \text{where } e_i \notin C \text{and } j \in G \setminus e_i[C] 
        \\
        \\
        (-1)^{d^*(j,u)}\frac{1}{\sqrt{{w_{uv}}}}\frac{|wt(\gamma_{j \rightarrow u})|}{wt(\gamma_{j \rightarrow u})} & \text{where } e_i \notin C \text{and } j \in G \setminus (e_i,u](C) 
        \\
        \\
        (-1)^{d^*(j,v)}\frac{1}{\sqrt{\overline{w_{uv}}}}\frac{|wt(\gamma_{j \rightarrow v})|}{wt(\gamma_{j \rightarrow v})} & \text{where } e_i \notin C \text{and } j \in G \setminus (e_i,v](C) 
    \end{cases}
    $$
    where $d^*(s,r)$ and $\gamma_{s \rightarrow r}$ denote the 
    distance and walk, respectively, between vertex $s$ and vertex $r$ in the underlying tree that is formed by removing $e_i = (u,v)$ from $G$. Then, $H$ is the Moore--Penrose inverse of $N$ where $N$ is the complex signless incidence matrix of $G$. 
}
\end{theorem}

\begin{proof}
    \color{blue}{
    It suffices to show $HN = I_n$ as Theorem \ref{thm:rank} implies that $H$ is invertible and thus $H = N^{-1}=N^+$. 
    \\
    \textbf{Case 1}: $i = j$.
    \\
    \textbf{Case 1a}: $e_i = (u,v) \in C$.
    $$(FN)_{ii} = f_{iu}\sqrt{w_{uv}} + f_{iv}\sqrt{\overline{w_{uv}}}$$
    $$= (-1)^{d^*(u,u)}\frac{\frac{wt(\gamma_{u \rightarrow u})}{|wt(\gamma_{u \rightarrow u})|}}{\sqrt{\overline{w_{uv}}}(\frac{w_{uv}}{|w_{uv}|}+\frac{wt(\gamma_{u \rightarrow v})}{|wt(\gamma_{u \rightarrow v})|})} \sqrt{w_{uv}} + (-1)^{d^*(u,v)}\frac{\frac{wt(\gamma_{u \rightarrow v})}{|wt(\gamma_{u \rightarrow v})|}}{\sqrt{\overline{w_{uv}}}(\frac{w_{uv}}{|w_{uv}|}+\frac{wt(\gamma_{u \rightarrow v})}{|wt(\gamma_{u \rightarrow v})|})} \sqrt{\overline{w_{uv}}}$$
    because $C$ is an odd cycle, $(-1)^{d^*(u,v)} = +1$. Thus, we get:
    $$= \frac{\frac{w_{uv}}{|w_{uv}|}}{\frac{w_{uv}}{|w_{uv}|}+\frac{wt(\gamma_{u \rightarrow v})}{|wt(\gamma_{u \rightarrow v})}} + \frac{\frac{wt(\gamma_{u \rightarrow v})}{|wt(\gamma_{u \rightarrow v})|}}{\frac{w_{uv}}{|w_{uv}|}+\frac{wt(\gamma_{u \rightarrow v})}{|wt(\gamma_{u \rightarrow v})|}} = 1$$
    \\
    \textbf{Case 1b}: $e_i = (u,v) \notin C$.
    \\
    The proof is identical to that of Case 1b. in Theorem \ref{mpeven}.
    \\
    \textbf{Case 2}: $i \neq j$
    \\
     Let $e_i = (u,v)$ and $e_j = (k,\ell)$.
    \\
    \textbf{Case 2a}: $e_i = (u,v) \in C$. 
    \\
    The proof is identical to that of Case 2a. in Theorem \ref{mpeven}.
    {
    \\
    }
    \\
    \textbf{Case 2b}: $e_i = (u,v) \notin C$. 
    \\
    \textbf{Case 2b.1}: $k,\ell \in G\setminus e_i[C]$. 
    \\
    The proof is identical to that of Case 2b.1 in Theorem \ref{mpeven}.
    {
    }
    \\
    \textbf{Case 2b.2}: $k,\ell \in G\setminus (e_i, v](C)$ or $k,\ell \in G\setminus (e_i, u](C)$.
    \\
    The proof is identical to that of Case 2b.2 in Theorem \ref{mpeven}.
    } 
\end{proof}

\section{Combinatorial Formula of \texorpdfstring{$\vol(N)$}{vol(N)} and \texorpdfstring{$N^+$}{N+} for any Weakly Connected Oriented Graph}

To find the combinatorial formula for the Moore--Penrose inverse $N^+$ of the signless incidence matrix $N$ of any weakly connected oriented graph with complex edge weights, we have the following theorem: 

\begin{theorem} \cite[Thm 2.6, 2.8]{MallikReddyVolume} \label{generalformula}Let $A$ be a $n \times m$ complex matrix with rank $r > 0.$ Then 

\[
A^+ = \frac{1}{\vol^2(A)} \sum_{(i_1, i_2,...,i_r) \in \mathcal{M}} \vol^2(A(i_1, i_2,..., i_r))A(i_1, i_2,..., i_r)^+, 
\]

where $\mathcal{M} = \{(i_1, i_2,...,i_r) \in \mathbb{Z}^+ | 1 \leq i_1 < i_2 < ... < i_r \leq m, \rank(A(i_1, i_2,...,i_r))=r\}$. 
\end{theorem}

Note that when $\rank(A) = r > 0$, the volume squared of $A$, which we will denote by $\vol^2(A)$, can be computed by summing over the squared modulus of the determinants of all $r \times r$ submatrices of $A$ \cite{MallikReddyVolume}. 

Now we will find the volume formulas for the signless incidence matrices of weakly connected oriented graphs using the volume definition from \cite{MallikReddyVolume}. 

\begin{theorem}
\label{thm:volumeoftree}
    Let $T$ be an oriented tree on $n \geq 2 $ vertices with complex edge weights and complex signless incidence matrix $N$. Then $\vol^2(N) = n |W_T|$.  

\begin{proof}
 Given that $T$ is an oriented tree, by Theorem~\ref{thm:rank}, $\rank(N) = n-1$. It follows that $\vol^2(N) = \sum^{n}_{i=1} |\det(N(i; ))|^2$, where $N(i; )$ denotes an $(n-1) \times (n-1)$ submatrix obtained by removing row $i =1,2,...,n$ in $N$. 

 Since every non-trivial tree has at least two leaf vertices, we may perform repeated cofactor expansions along the rows corresponding to such vertices. This process yields 

\[
\det(N(i;)) = \pm \prod_{(u,v) \in E(T)}\sqrt{a_{uv}}
\]
where $a_{uv} \in \{w_{uv},\overline{w_{uv}}\}$ for each edge $(u,v)$ in $T$. 

Taking the modulus of both sides:
\[
|\det(N(i;))| = \prod_{(u,v) \in E(T)} |\sqrt{a_{uv}}| = \prod_{(u,v) \in E(T)} |\sqrt{w_{uv}}| = |\sqrt{W_T}|. 
\]

Therefore, 

\[
\vol^2(N) = \sum^{n}_{i=1} |\det(N(i; ))|^2 = \sum^{n}_{i=1} |W_T| = n|W_T|. 
\]
    
\end{proof}

\end{theorem}

\begin{example}
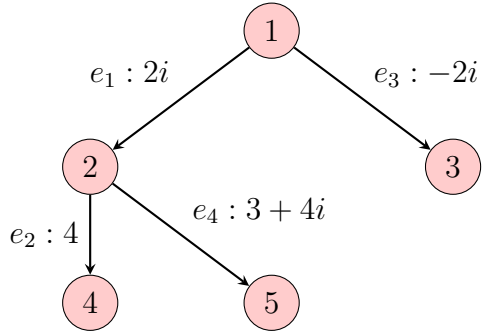
\begin{figure}[htbp]
\centering
\begin{tikzpicture}[scale=1.2,
                    colorstyle/.style={circle, fill, black, scale=.5},
                    >=stealth]
\tikzset{vertex/.style = {shape = circle, draw, minimum size = 1em, fill=pink!80}}
\tikzset{edge/.style = {->,thick}}

    \node[vertex] (1) at (0,3) {$1$};
    \node[vertex] (2) at (-2,1.5) {$2$};
    \node[vertex] (3) at (2,1.5) {$3$};
    \node[vertex] (4) at (-2,0) {$4$};
    \node[vertex] (5) at (0,0) {$5$};

    \draw[edge] (1) edge node[midway, above left] {$e_1:2i$} (2);
    \draw[edge] (2) edge node[midway, left] {$e_2:4$} (4);
    \draw[edge] (1) edge node[midway, above right] {$e_3:-2i$} (3);
    \draw[edge] (2) edge node[midway, above right] {$e_4:3+4i$} (5);

    \end{tikzpicture}
    \caption{Example of an oriented tree $T$ with complex edge weights}.
    \end{figure}

    We know that the complex signless incidence matrix of $T$ is as follows: 

    \[
    N = 
    \begin{bmatrix}
        1+i & 1-i & 0 & 0  \\
        1-i & 0   & 2 & 2+i \\
        0   & 1+i & 0 & 0   \\
        0   & 0   & 2 & 0 \\
        0   & 0   & 0 & 2-i
    \end{bmatrix}
\]

According to SageMath, $\vol^2(N) = 20.$ 

Theorem~\ref{thm:volumeoftree} confirms this value as 
\[
\vol^2(N) = 5|2i||4||-2i||3+4i| = 400. 
\]
Hence, $\vol(N) = \sqrt{400} = 20. $
\end{example}

\begin{theorem}
\label{thm:volumeofcycle}
    \textcolor{blue}{Let $H$ be an oriented cycle on $n \geq 2$ vertices with non-zero edge weights in $\C \setminus \R^-$ and complex signless incidence matrix $N$.
    \\ If $n$ is even, then $|\det(N)| = \sqrt{2(|wt(H)| - \Re(wt(H)))}$. 
    \\
    \indent (a) If $wt(H) \notin \R^+, \text{then } \det(N) \neq 0 \text{ and } \vol^2(N) = |\det(N)|^2 = 2(|wt(H)| - \Re(wt(H)))$.
    \indent (b) If $wt(H) \in \R^+, \text{then } \det(N) = 0 \text{ and } \vol^2(N) = n|W_H| \sum\limits_{(u,v)\in E(H)} \frac{1}{|w_{uv}|}$.
    \\
    If $n$ is odd, then $|\det(N)| = \sqrt{2(|wt(H)| + \Re(wt(H)))}$.
    \\
    \indent (c) If $wt(H) \notin \R^-, \text{then } \det(N) \neq 0 \text{ and } \vol^2(N) = |\det(N)|^2 = 2(|wt(H)| + \Re(wt(H)))$.
    \indent (d) If $wt(H) \in \R^-, \text{then } \det(N) = 0 \text{ and } \vol^2(N) = n|W_H| \sum\limits_{(u,v)\in E(H)} \frac{1}{|w_{uv}|}$.}
\end{theorem}

\begin{proof}
    \textcolor{blue}{
    For all $i \in 1,2,...,n$, let $e_i = (u_i, v_i)$ and $a_i \in \{ \sqrt{w_{u_iv_i}},\sqrt{\overline{w_{u_iv_i}}} \}$ such that $a_1a_2...a_n = \sqrt{wt(H)}$. 
    \\
    We can factorize $N$ as $N = PN'Q$ where $P$ and $Q$ are both permutation matrices and $$N' = 
    \begin{bmatrix}
    a_1 & 0 & 0 & \cdots & 0 & \overline{a_n} \\
    \overline{a_1} & a_2 & 0 & \cdots & 0 & 0 \\
    0 & \overline{a_2} & a_3 & \ddots & \vdots & \vdots \\
    \vdots & \ddots & \ddots & \ddots & \ddots & \vdots \\
    \vdots & \vdots & \ddots & \ddots & a_{n-1} & 0 \\
    0 & \cdots & \cdots & 0 & \overline{a_{n-1}} & a_n
    \end{bmatrix}
    $$
    By doing cofactor expansion across the first row, it then follows that:
    $$\det(N) = \pm1\left( a_1a_2...a_n + (-1)^{n-1}\overline{a_na_1...a_{n-1}} \right)$$
    Thus, if $n$ is even, we can deduce $\det(N) = \pm1\left( a_1a_2...a_n - \overline{a_na_1...a_{n-1}} \right) = \pm2i\Im(\sqrt{wt(H)})$ which implies:
    $$|\det(N)|^2 = \left|2i\Im(\sqrt{wt(H)}\right|^2 = 2\left( |wt(H)| - \Re(wt(H)) \right)$$
    Therefore, $|\det(N)| = \sqrt{2(|wt(H)| - \Re(wt(H)))}$. 
    \\
    \indent (a) If $wt(H) \notin \R^+, \text{then } \det(N) \neq 0 \text{, so by Proposition 1.1 \cite{MallikReddyVolume}, } \vol^2(N) = |\det(N)|^2 = 2(|wt(H)| - \Re(wt(H)))$.
    \\
    Thus, if $n$ is odd, we can deduce $\det(N) = \pm1\left( a_1a_2...a_n + \overline{a_na_1...a_{n-1}} \right) = \pm2\Re(\sqrt{wt(H)})$ which implies:
    $$|\det(N)|^2 = \left|2\Re(\sqrt{wt(H)}\right|^2 = 2\left( |wt(H)| + \Re(wt(H)) \right)$$
    Therefore, $|\det(N)| = \sqrt{2(|wt(H)| + \Re(wt(H)))}$. 
    \\
    \indent (c) If $wt(H) \notin \R^-, \text{then } \det(N) \neq 0 \text{ , so by Proposition 1.1 \cite{MallikReddyVolume}, } \vol^2(N) = |\det(N)|^2 = 2(|wt(H)| + \Re(wt(H)))$.
    \\
    \indent For (b) and (d), we know $\rank(N) \leq n-1$, so we can consider $N(:j)$ which is the corresponding signless incidence matrix for when we delete edge $j$, $e_j = (u,v)$, from $H$. Since $H \setminus e_j$ is a tree, by Theorem \ref{thm:volumeoftree}, we have $|\det(N(i:j))| = \sqrt{|W_{H\setminus e_j}|} = \sqrt{\frac{|W_H|}{|w_{uv}|}}$ which implies $\rank(N) = n-1$. By Theorem 2.4 \cite{MallikReddyVolume}, if we let $i$ index the vertices and $j$ index the edges, we therefore have:
    $$\vol^2(N) = \sum_{i=1}^n\sum_{j=1}^n|\det(N(i:j))|^2 = \sum_{i=1}^n\sum_{j=1}^n {\frac{|W_H|}{|w_{uv}|}} = n{|W_H|}\sum_{j=1}^n {\frac{1}{|w_{uv}|}} = n|W_H| \sum\limits_{(u,v)\in E(H)} \frac{1}{|w_{uv}|}$$
    }
\end{proof}

\begin{lemma}
\label{rankofdisconnectedgraph}
    Let $G$ be a disconnected oriented graph on $n$ vertices with cycles such that $wt(C) \in \mathbb{R^+}$ for each even cycle $C$ and $wt(C) \in \mathbb{R^-}$ for each odd cycle $C$. Then the rank of its complex signless incidence matrix, denoted by $\rank(N)$, is $n-k$, with $k$ as the number of components in $G$. 
\end{lemma}

\begin{proof}
    Let $N_j$ denote the signless incidence matrix of an arbitrary component in $G$. By Theorem~\ref{thm:rank}, there exists a nonzero vector $x_j$ such that $N_j^Tx_j = 0$. Since there are $k$ linearly independent vectors $x_j$, they form a complete basis for the left nullspace of $N$. Thus, by the Rank-Nullity Theorem, $\rank(N) = n -k$. 
\end{proof}

\begin{observation}
\label{observation}
    Let $G$ be a graph on $n$ vertices with $n$ edges. If $G$ is connected, then $G$ must be a unicyclic graph. If $G$ is disconnected, then either each component of $G$ has equal number of edges and vertices, or at least one component has fewer edges than vertices. 
\end{observation}

\begin{theorem}\label{volformulas}
    Let $G$ be a weakly connected oriented graph on $n \geq 2$ vertices with oriented cycles and complex edge weights and complex signless incidence matrix $N$. 

    (a) If $wt(C) \in \mathbb{R^+}$ for every even cycle $C$ in $G$ and $wt(C) \in \mathbb{R^-}$ for every odd cycle $C$ in $G$, then 

    \[
    \vol^2(N) = n \sum_{T \in \mathcal{T}(G)} |W_T|, 
    \]

    where $\mathcal{T}(G)$ denotes the set of all oriented spanning trees $T$ of $G$.

    (b) If $G$ contains at least one even cycle $C$ whose $wt(C) \notin \mathbb{R}^+$ or at least one odd cycle $C$ whose $wt(C) \notin \mathbb{R}^-$, then 

    \[
    \vol^2(N) = \sum_{H \in \mathcal{U}(G)} \prod_{i=1}^{c(H)} 2|W_{U_i\setminus C_i}|
    \cdot \alpha(C_i)
    \]

   where $\mathcal{U}(G)$ denotes the set of all spanning subgraphs $H$ of $G$ consisting of $c(H)$ unicyclic components $U_1, U_2, ..., U_{c(H)}$ and 

   \[
   \alpha(C_i) = \begin{cases}
       |wt(C_i)| - \Re(wt(C_i))  & \text{if $C_i$ is an even cycle} \\
       |wt(C_i)| + \Re(wt(C_i))  & \text{if $C_i$ is an odd cycle} 
   \end{cases}
   \]
\end{theorem}

\begin{proof}
    (a) Since $wt(C) \in \mathbb{R^+}$ for every even cycle $C$ in $G$ and $wt(C) \in \mathbb{R^-}$ for every odd cycle $C$ in $G$, Theorem~\ref{thm:rank} implies that $\rank(N) = n-1$. Hence, 

    \[
    \vol^2(N) = \sum^{n}_{i=1} \sum_{S \subseteq E(G), |S| = n-1} |\det(N(i;S])|^2, 
    \]
    
    where $N(i;S]$ denotes the $(n-1) \times (n-1)$ submatrix obtained by removing row $i =1,2,...,n$ of $N$ and selecting a subset of columns $S$ of size $n-1$. 

    For each $S \subseteq E(G), |S| = n-1$, we can define a spanning subgraph $H$ of $G$ such that 

    \[
    H := (V(G), S). 
    \]

    Then $N(i;S]$ is the complex signless incidence matrix of $H$ with one row removed. 
    
    We consider the following two cases.
    \\
    Case 1: H is connected 
    
    If $H$ is connected, then it is a spanning tree. By Theorem~\ref{thm:volumeoftree}, $|\det(N(i;S])| = \sqrt{|W_H|}$. 
    \\
    Case 2: H is disconnected 
    
    Since $H$ has at least two disconnected components, by Lemma~\ref{rankofdisconnectedgraph}, the rank of its incidence matrix $N(;S]$ must be at most $n-2$. Removing a row will not increase the rank of a matrix so $\rank(N(i;S]) \leq n-2 < n-1$, which implies $N(i;S]$ is not full rank. Consequently, $\det(N(i;S]) = 0$. 
    \\

    (b) Since $G$ contains at least one even cycle $C$ whose $wt(C) \notin \mathbb{R}^+$ or at least one odd cycle $C$ whose $wt(C) \notin \mathbb{R}^-$, by Theorem~\ref{thm:rank}, $\rank(N) = n$. By Binet-Cauchy Theorem \cite{MallikReddyVolume} we have:

    \[
    \vol^2(N) = \det(NN^*) = \sum_{S \subseteq \{1,\ldots,m\},\ |S| = n} |\det(N(;S])|^2
    \]

    where $N(;S]$ denotes an $n \times n$ submatrix formed by keeping all vertices and selecting a subset of exactly $n$ edges chosen from the total $m$ edges in $G$. 

    Similar to the proof in part a, for each $S \subseteq E(G), |S| = n$, we can define a spanning subgraph $H$ of $G$ such that 

    \[
    H := (V(G), S)
    \]

    with $N(;S]$ as the complex signless incidence matrix of $H$. 
    
    We consider the following two cases.
    \\
    \textbf{Case 1}: $H$ is connected 
    
    If $H$ is connected on $n$ vertices with $n$ edges, then $H$ must be a unicyclic graph with oriented cycle $C$. Let $H \setminus C$ be the oriented forest obtained by removing the edges of $C$. 
    
    Then, by sucessive cofactor expansions along the rows corresponding to the leaf vertices of $H$, we have

    \[
    \det(N(;S]) = \prod_{(u,v) \in E(H \setminus C)} \sqrt{a_{uv}} \cdot \det(N_C) 
    \], 
    
    where $a_{uv} \in \{w_{uv},\overline{w_{uv}}\}$ for each edge $(u,v)$ in $H \setminus C$.

    Consequently,  
    \[
    |\det(N(;S])|^2 = \prod_{(u,v) \in E(H \setminus C)} |w_{uv}| \cdot |\det(N_C)|^2 = |W_{H \setminus C}|  \cdot |\det(N_C)|^2
    \]
    
    By Theorem~\ref{thm:volumeofcycle}, if the cycle $C$ is even and $wt(C) \in \mathbb{R}^+$ or if $C$ is odd and $wt(C) \in \mathbb{R}^-$, then $|\det(N(;S])|^2 = 0$. 
    
    On the other hand, if $C$ is even and $wt(C) \notin \mathbb{R}^+$, then $|\det(N(;S])|^2 = 2|W_{H \setminus C}|(|wt(C)| - \Re(wt(C)))$. Similarly, if $C$ is odd and $wt(C) \notin \mathbb{R}^-$, then $|\det(N(;S])|^2 = 2|W_{H \setminus C}|(|wt(C)| + \Re(wt(C)))$. 
    \\
    \textbf{Case 2}: $H$ is disconnected 

    By Observation~\ref{observation}, suppose $H$ is disconnected and at least one component has fewer edges than vertices. We can permute its signless incidence matrix $N(;S]$ into a block matrix where at least one block has more rows than columns. This implies a linear dependency among the rows and forces $\det(N(;S]) = 0$. Consequently, only spanning subgraphs $H$ with disjoint unicyclic components contribute to the volume of $N$. 

    Let $U_1, U_2, ..., U_{c(H)}$ be disjoint unicyclic components of $H$. By permuting the rows and columns of $N(;S]$, we can force $N(;S]$ be a block-diagonal matrix with each block corresponding to a unicyclic component. We have: 

    \[
    |\det(N(;S])|^2 = \prod_{i=1}^{c(H)} |\det(N_{U_i})|^2
    \]

    From here, the proof is similar to case 1. Hence, 

     \[
    \vol^2(N) = \sum_{H \in \mathcal{U}(G)} \prod_{i=1}^{c(H)} 2|W_{U_i\setminus C_i}|
    \cdot \alpha(C_i)
    \]

   where $\mathcal{U}(G)$ denotes the set of all spanning subgraphs $H$ of $G$ consisting of $c(H)$ unicyclic components $U_1, U_2, ..., U_{c(H)}$ and 

   \[
   \alpha(C_i) = \begin{cases}
       |wt(C_i)| - \Re(wt(C_i))  & \text{if $C_i$ is an even cycle} \\
       |wt(C_i)| + \Re(wt(C_i))  & \text{if $C_i$ is an odd cycle} 
   \end{cases}
   \]

\end{proof}

\begin{theorem}\label{final1}
    \textcolor{blue}{Let $G$ be a weakly connected oriented graph on $n \geq 2$ vertices $1,2,..n$ and $m \geq 1$ edges $1,2,..m$ with non-zero edge weights in $\C \setminus \R^-$ and complex signless incidence matrix $N$ such that $G$ contains oriented cycles and $wt(C) \in \R^-$ for all odd cycles and $wt(C) \in \R^+$ for all even cycles. Then, $\rank(N) = n-1$ and 
    $$(N^+)_{ij} = \frac{1}{ \sum_{T \in \mathcal{T}(G)} |W_T|}\sum_{T \in \mathcal{T}(G), e_i \in T}|W_T|(N_T^+)_{ij}$$
    where $\mathcal{T}(G)$ is the set of all oriented spanning trees $T$ of $G$ with $e_i = (u,v) \in T$ and:
    $$
        (N_T^+)_{ij} = \frac{(-1)^{d(e_i,j)}}{n}
    \begin{cases}
        \frac{|T_h(e_i)|}{\sqrt{w_{uv}}}\frac{|wt(\gamma_{j \rightarrow u})|}{wt(\gamma_{j \rightarrow u})} & \text{where } j \in T_t(e_i)
        \\
        \\
        \frac{|T_t(e_i)|}{\sqrt{\overline{w_{uv}}}}\frac{|wt(\gamma_{j \rightarrow v})|}{wt(\gamma_{j \rightarrow v})} & \text{where } j \in T_h(e_i) 
    \end{cases}
    $$
    }
\end{theorem}

\begin{proof}
    \textcolor{blue}{Using Theorem \ref{thm:rank}, we deduce that $\rank(N) = n-1$. By Theorem \ref{generalformula}, we need only to consider the spanning trees of $G$ for if there existed a subgraph on $n$ vertices and $n-1$ edges where a connected component contained at least one cycle, then, $\rank(N) < n-1$. In order to find $\vol^2(N(i_1,..,i_{n-1}))$, since $\vol^2(N(i_1,..,i_{n-1})) = \vol^2(N)$ for the subgraph with the vertices $V(G)$ and edges $\{ e_{i_1},..,e_{i_{n-1}} \} \subset E(G)$, we can use Theorem \ref{thm:volumeoftree}. In order to find $N^+(i_1,..,i_{n-1})$, we can use Theorem \ref{mptree}. Finally, using Theorem \ref{volformulas} (a) we have $\vol^2(N)$ for the $N$ corresponding to $G$.
    }
\end{proof}

\begin{theorem}\label{final2}
    \textcolor{blue}{Let $G$ be a weakly connected oriented graph on $n \geq 2$ vertices $1,2,..n$ and $m \geq 1$ edges $1,2,..m$ with non-zero edge weights in $\C \setminus \R^-$ and complex signless incidence matrix $N$ such that $G$ contains oriented cycles and there exists $wt(C) \in \C \setminus \R^-$ for an odd cycle or $wt(C) \in \C \setminus \R^+$ for an even cycle. Then, $\rank(N) = n$ and 
    $$(N^+)_{ij} = \frac{1}{\vol^2(N)}\sum_{H \in \mathcal{U}(G), e_i \in H} (N_H^+)_{ij} \prod_{i = 1}^{c(H)}2|W_{U_i\setminus C_i}|
    \cdot \alpha(C_i)
    $$
    $$
    \vol^2(N) = \sum_{H \in \mathcal{U}(G)} \prod_{i=1}^{c(H)} 2|W_{U_i\setminus C_i}|
    \cdot \alpha(C_i)
    $$
    $$
   \alpha(C_i) = \begin{cases}
       |wt(C_i)| - \Re(wt(C_i))  & \text{if $C_i$ is an even cycle} \\
       |wt(C_i)| + \Re(wt(C_i))  & \text{if $C_i$ is an odd cycle} 
   \end{cases}
   $$
   where $\mathcal{U}(G)$ denotes the set of all spanning subgraphs $H$ of $G$ on $n$ edges consisting of $c(H)$ unicyclic components $U_1, U_2, ..., U_{c(H)}$ with $wt(C) \in \C \setminus \R^+$ for all even cycles and  $wt(C) \in \C \setminus \R^-$ for all odd cycles and when $e_i$ is in $U_r$ for some r. Also, $(N_H^+)_{ij}$ is given by:
   \\
   If $e_i$ is in $U_r$ with even cycle $C_r$, then 
    $$(N^+_{H})_{ij} = 
    \begin{cases}
        (-1)^{d^*(j,u)}\frac{\frac{wt(\gamma_{u \rightarrow j})}{|wt(\gamma_{u \rightarrow j})|}}{\sqrt{\overline{w_{uv}}}(\frac{w_{uv}}{|w_{uv}|}-\frac{wt(\gamma_{u \rightarrow v})}{|wt(\gamma_{u \rightarrow v})|})} & \text{where } e_i = (u,v) \in C_r
        \\
        \\
        0 & \text{where } e_i \notin C \text{and } j \in G \setminus e_i[C_r] 
        \\
        \\
        (-1)^{d^*(j,u)}\frac{1}{\sqrt{{w_{uv}}}}\frac{|wt(\gamma_{j \rightarrow u})|}{wt(\gamma_{j \rightarrow u})} & \text{where } e_i \notin C \text{and } j \in G \setminus (e_i,u](C_r) 
        \\
        \\
        (-1)^{d^*(j,v)}\frac{1}{\sqrt{\overline{w_{uv}}}}\frac{|wt(\gamma_{j \rightarrow v})|}{wt(\gamma_{j \rightarrow v})} & \text{where } e_i \notin C \text{and } j \in G \setminus (e_i,v](C_r) 
        \\
        0 & \text{where } j \notin U_r
    \end{cases}
    $$
    If $e_i$ is in $U_r$ with odd cycle $C_r$, then
    $$(N^+_{H})_{ij} = 
    \begin{cases}
        (-1)^{d^*(j,u)}\frac{\frac{wt(\gamma_{u \rightarrow j})}{|wt(\gamma_{u \rightarrow j})|}}{\sqrt{\overline{w_{uv}}}(\frac{w_{uv}}{|w_{uv}|}+\frac{wt(\gamma_{u \rightarrow v})}{|wt(\gamma_{u \rightarrow v})|})} & \text{where } e_i = (u,v) \in C_r
        \\
        \\
        0 & \text{where } e_i \notin C \text{and } j \in G \setminus e_i[C_r] 
        \\
        \\
        (-1)^{d^*(j,u)}\frac{1}{\sqrt{{w_{uv}}}}\frac{|wt(\gamma_{j \rightarrow u})|}{wt(\gamma_{j \rightarrow u})} & \text{where } e_i \notin C \text{and } j \in G \setminus (e_i,u](C_r) 
        \\
        \\
        (-1)^{d^*(j,v)}\frac{1}{\sqrt{\overline{w_{uv}}}}\frac{|wt(\gamma_{j \rightarrow v})|}{wt(\gamma_{j \rightarrow v})} & \text{where } e_i \notin C \text{and } j \in G \setminus (e_i,v](C_r) 
        \\
        0 & \text{where } j \notin U_r
    \end{cases}
    $$
    where $d^*(s,r)$ and $\gamma_{s \rightarrow r}$ denote the 
    distance and walk, respectively, between vertex $s$ and vertex $r$ in the underlying tree that is formed by removing $e_i$ from $G$. }
\end{theorem}

\begin{proof}
    \textcolor{blue}{Using Theorem \ref{thm:rank}, we deduce that $\rank(N) = n$. By Theorem \ref{generalformula}, we need only to consider the spanning subgraphs of $G$ that consist only of unicyclic connected components where each unicyclic component is full rank (i.e. $wt(C) \in \C \setminus \R^+$ for all even cycles and  $wt(C) \in \C \setminus \R^-$ for all odd cycles) for if there existed a subgraph on $n$ vertices and $n$ edges where at least one connected component contained a tree, then, $\rank(N) < n$, and if a connected component contained at least one multicyclic component, we cannot have $n$ vertices and $n$ edges. 
    \\
    Note that $N_H = P_1N'P_2$ where $P_1$ and $P_2$ are permutation matrices and $N'$ is a block diagnol matrix of $N_{U_r}$ for all ${U_r}$ in $H$.
    To find $\vol^2(N(i_1,..,i_{n}))$, since $\vol^2(N(i_1,..,i_{n})) = \vol^2(N) = \vol^2(P_1N'P_2) = \vol^2(N') = \prod_{r=1}^{c(H)}\vol^2(N_{U_r})$  for N corresponding to the subgraph of $G$ with vertices $V(G)$ and edges $\{e_{i_1},..,e_{i_{n}} \} \subset E(G)$. We can use  Theorem \ref{volformulas} (b) to find $\vol^2(N_{U_r})$. In order to find $N^+(i_1,..,i_{n-1})$, we can use Theorem \ref{mpeven} and Theorem \ref{mpodd}. Finally, using Theorem \ref{volformulas} (b) we have $\vol^2(N)$ for the $N$ corresponding to $G$.
    }
\end{proof}

\begin{example}
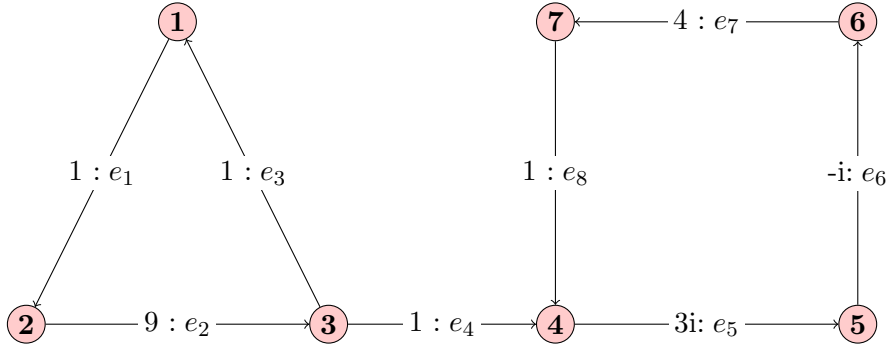
\begin{figure}[htbp]
\centering
\begin{tikzpicture}[
    node distance=2cm and 3cm,
    vertex/.style={
        circle, 
        draw=black, 
        fill=pink!80, 
        minimum size=5mm, 
        inner sep=0pt,
        font=\small\bfseries
    },
    edge/.style={
        ->, 
        >={Stealth[scale=2]}, 
        thick, 
        shorten >=1pt
    },
    weight/.style={
        font=\small\color{blue}, 
        fill=white, 
        inner sep=2pt
    }
    ]

    \node[vertex] (3) at (4,0) {3};
    \node[vertex] (1) at (2,4) {1};
    \node[vertex] (2) at (0,0) {2};
    \node[vertex] (4) at (7,0) {4};
    \node[vertex] (5) at (11,0) {5};
    \node[vertex] (6) at (11,4) {6};
    \node[vertex] (7) at (7,4) {7};

    \draw[<-] (2) -- node[weight] {1 $: e_1$} (1);
    \draw[<-] (1) -- node[weight] {1 $: e_3$} (3);
    \draw[<-] (4) -- node[weight] {1 $: e_4$} (3);
    \draw[<-] (3) -- node[weight] {9 $: e_2$} (2);
    \draw[<-] (5) -- node[weight] {3\i $: e_5$} (4);
    \draw[<-] (6) -- node[weight] {-\i $: e_6$} (5);
    \draw[<-] (7) -- node[weight] {4 $: e_7$} (6);
    \draw[<-] (4) -- node[weight] {1 $: e_8$} (7);

    \end{tikzpicture}
    \caption{\color{blue}{Example of a graph $G$ such that $G$ has an odd cycle, $C$,  with $wt(C) \in \C \setminus \R^-$}.}\label{finalex2}
    \end{figure}
    \color{blue}{Using Sage math and rounding to the nearest fifth digit if applicable, we get that for the graph in Figure \ref{finalex2}, $(N^+)_{23} = 0.16667$, $(N^+)_{35} = 0.5\i$, $(N^+)_{42} = 0$, $(N^+)_{56} = -0.19754 + 0.19754\i$, $(N^+)_{85} = 0.12903\i$. We will use Theorem \ref{final2} to confirm these values.}

\begin{figure}[htbp]
\begin{minipage}[t]{0.48\textwidth}
\begin{tikzpicture}[
scale=0.65, 
every node/.append style={transform shape}, 
    node distance=2cm and 3cm,
    vertex/.style={
        circle, 
        draw=black, 
        fill=pink!80, 
        minimum size=5mm, 
        inner sep=0pt,
        font=\small\bfseries
    },
    edge/.style={
        ->, 
        >={Stealth[scale=2]}, 
        thick, 
        shorten >=1pt
    },
    weight/.style={
        font=\small\color{blue}, 
        fill=white, 
        inner sep=2pt
    }
    ]

    \node[vertex] (3) at (4,0) {3};
    \node[vertex] (1) at (2,4) {1};
    \node[vertex] (2) at (0,0) {2};
    \node[vertex] (4) at (7,0) {4};
    \node[vertex] (5) at (11,0) {5};
    \node[vertex] (6) at (11,4) {6};
    \node[vertex] (7) at (7,4) {7};

    \draw[<-] (2) -- node[weight] {1 $: e_1$} (1);
    \draw[<-] (1) -- node[weight] {1 $: e_3$} (3);
    \draw[<-] (4) -- node[weight] {1 $: e_4$} (3);
    \draw[<-] (3) -- node[weight] {9 $: e_2$} (2);
    \draw[<-] (6) -- node[weight] {-\i $: e_6$} (5);
    \draw[<-] (7) -- node[weight] {4 $: e_7$} (6);
    \draw[<-] (4) -- node[weight] {1 $: e_8$} (7);
    \end{tikzpicture}
\end{minipage}
\begin{minipage}[t]{0.48\textwidth}
\begin{tikzpicture}[
scale=0.65, 
every node/.append style={transform shape}, 
    node distance=2cm and 3cm,
    vertex/.style={
        circle, 
        draw=black, 
        fill=pink!80, 
        minimum size=5mm, 
        inner sep=0pt,
        font=\small\bfseries
    },
    edge/.style={
        ->, 
        >={Stealth[scale=2]}, 
        thick, 
        shorten >=1pt
    },
    weight/.style={
        font=\small\color{blue}, 
        fill=white, 
        inner sep=2pt
    }
    ]

    \node[vertex] (3) at (4,0) {3};
    \node[vertex] (1) at (2,4) {1};
    \node[vertex] (2) at (0,0) {2};
    \node[vertex] (4) at (7,0) {4};
    \node[vertex] (5) at (11,0) {5};
    \node[vertex] (6) at (11,4) {6};
    \node[vertex] (7) at (7,4) {7};

    \draw[<-] (2) -- node[weight] {1 $: e_1$} (1);
    \draw[<-] (1) -- node[weight] {1 $: e_3$} (3);
    \draw[<-] (4) -- node[weight] {1 $: e_4$} (3);
    \draw[<-] (3) -- node[weight] {9 $: e_2$} (2);
    \draw[<-] (5) -- node[weight] {3\i $: e_5$} (4);
    \draw[<-] (7) -- node[weight] {4 $: e_7$} (6);
    \draw[<-] (4) -- node[weight] {1 $: e_8$} (7);
    \end{tikzpicture}
\end{minipage}

\begin{tikzpicture}[
scale=0.65, 
every node/.append style={transform shape}, 
    node distance=2cm and 3cm,
    vertex/.style={
        circle, 
        draw=black, 
        fill=pink!80, 
        minimum size=5mm, 
        inner sep=0pt,
        font=\small\bfseries
    },
    edge/.style={
        ->, 
        >={Stealth[scale=2]}, 
        thick, 
        shorten >=1pt
    },
    weight/.style={
        font=\small\color{blue}, 
        fill=white, 
        inner sep=2pt
    }
    ]

    \node[vertex] (3) at (4,0) {3};
    \node[vertex] (1) at (2,4) {1};
    \node[vertex] (2) at (0,0) {2};
    \node[vertex] (4) at (7,0) {4};
    \node[vertex] (5) at (11,0) {5};
    \node[vertex] (6) at (11,4) {6};
    \node[vertex] (7) at (7,4) {7};

    \draw[<-] (2) -- node[weight] {1 $: e_1$} (1);
    \draw[<-] (1) -- node[weight] {1 $: e_3$} (3);
    \draw[<-] (4) -- node[weight] {1 $: e_4$} (3);
    \draw[<-] (3) -- node[weight] {9 $: e_2$} (2);
    \draw[<-] (5) -- node[weight] {3\i $: e_5$} (4);
    \draw[<-] (6) -- node[weight] {-\i $: e_6$} (5);
    \draw[<-] (4) -- node[weight] {1 $: e_8$} (7);
    \end{tikzpicture}
\begin{tikzpicture}[
scale=0.65, 
every node/.append style={transform shape}, 
    node distance=2cm and 3cm,
    vertex/.style={
        circle, 
        draw=black, 
        fill=pink!80, 
        minimum size=5mm, 
        inner sep=0pt,
        font=\small\bfseries
    },
    edge/.style={
        ->, 
        >={Stealth[scale=2]}, 
        thick, 
        shorten >=1pt
    },
    weight/.style={
        font=\small\color{blue}, 
        fill=white, 
        inner sep=2pt
    }
    ]

    \node[vertex] (3) at (4,0) {3};
    \node[vertex] (1) at (2,4) {1};
    \node[vertex] (2) at (0,0) {2};
    \node[vertex] (4) at (7,0) {4};
    \node[vertex] (5) at (11,0) {5};
    \node[vertex] (6) at (11,4) {6};
    \node[vertex] (7) at (7,4) {7};

    \draw[<-] (2) -- node[weight] {1 $: e_1$} (1);
    \draw[<-] (1) -- node[weight] {1 $: e_3$} (3);
    \draw[<-] (4) -- node[weight] {1 $: e_4$} (3);
    \draw[<-] (3) -- node[weight] {9 $: e_2$} (2);
    \draw[<-] (5) -- node[weight] {3\i $: e_5$} (4);
    \draw[<-] (6) -- node[weight] {-\i $: e_6$} (5);
    \draw[<-] (7) -- node[weight] {4 $: e_7$} (6);
    \end{tikzpicture}
    \caption{\color{blue}{These are all possible $H$ of the graph $G$ shown in Figure \ref{finalex2}.}\label{finalex2H}}
\end{figure}
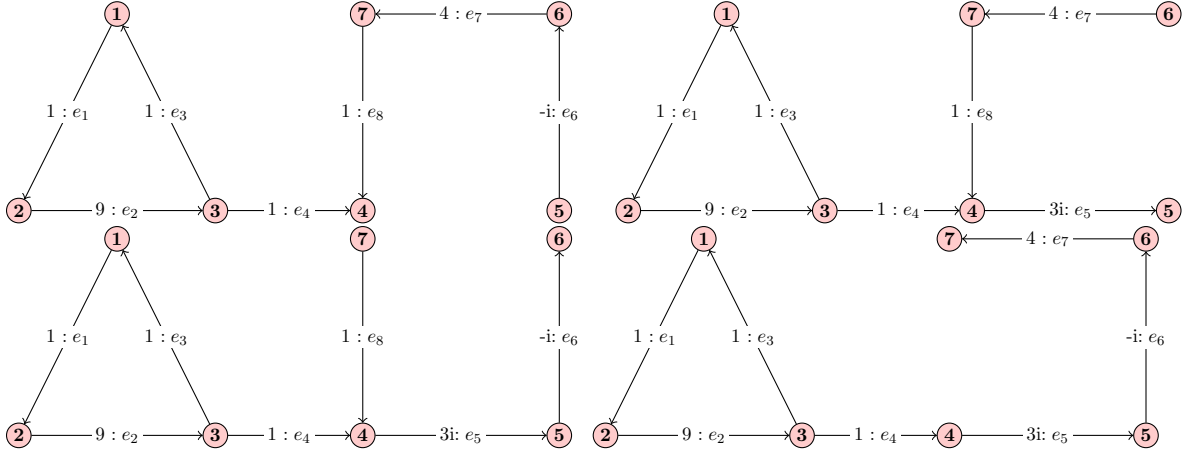

We get that $\vol^2(N) = 2(4)(18) + 2(12)(18) + 2(3)(18) + 2(12)(18) = 144 + 432 + 108 + 432 = 1116$ which agrees with Sage Math. 

\begin{align*}
(N^+)_{23} ={} &\frac{1}{1116} \Biggl[ (-1)^{d^*(3,2)}\frac{\frac{wt(\gamma_{2 \rightarrow 3})}{|wt(\gamma_{2 \rightarrow 3})|}}{\sqrt{9}(1+1)}144 
  + (-1)^{d^*(3,2)}\frac{\frac{wt(\gamma_{2 \rightarrow 3})}{|wt(\gamma_{2 \rightarrow 3})|}}{\sqrt{9}(1+1)}432 \\
& + (-1)^{d^*(3,2)}\frac{\frac{wt(\gamma_{2 \rightarrow 3})}{|wt(\gamma_{2 \rightarrow 3})|}}{\sqrt{9}(1+1)}108 
  + (-1)^{d^*(3,2)}\frac{\frac{wt(\gamma_{2 \rightarrow 3})}{|wt(\gamma_{2 \rightarrow 3})|}}{\sqrt{9}(1+1)}432 \Biggr] \\
= {} & \frac{1116}{6 \cdot 1116} = 0.16667 \quad 
\end{align*}

$$(N^+)_{42} = 0 $$
$$(N^+)_{56} = \frac{1}{1116} \left[(-1)^{d^*(6,5)}\frac{1}{\sqrt{-3\i}}\frac{1}{\i}108 +  (-1)^{d^*(6,5)}\frac{1}{\sqrt{-3\i}}\frac{1}{\i}432 \right]= \frac{-540}{1116 \cdot \i\sqrt{-3\i}} = -0.19754 + 0.19754\i $$
$$(N^+)_{85} = \frac{1}{1116} \left[(-1)^{d^*(5,7)}\frac{-1}{\i}144 \right] = \frac{-144}{1116\i} =0.12903\i $$
\end{example}

\section{Combinatorial Formulas of \texorpdfstring{$Q^+$}{Q+} and \texorpdfstring{$\det(Q)$}{det(Q)} for any Weakly Connected Oriented Graph}

\begin{theorem}
    \color{blue}{Let $G$ be a weakly connected oriented graph on $n \geq 2$ vertices $1,2,..n$ and $m \geq 1$ edges $1,2,..m$ with non-zero edge weights in $\C \setminus \R^-$ and complex signless Laplacian matrix $Q$. Then, 
    $$\det(Q) = \sum_{H \in \mathcal{U}(G)}  \prod_{i=1}^{c(H)} 2|W_{U_i\setminus C_i}|
    \cdot \alpha(C_i)$$
    
   \[
   \alpha(C_i) = \begin{cases}
       |wt(C_i)| - \Re(wt(C_i))  & \text{if $C_i$ is an even cycle} \\
       |wt(C_i)| + \Re(wt(C_i))  & \text{if $C_i$ is an odd cycle} 
   \end{cases}
   \]
   where $\mathcal{U}(G)$ is the set of all spanning subgraphs of $G$ on $n$ vertices and $n$ edges where the connected components of $H$ are all unicyclic graphs such that $wt(C) \in \C \setminus \R^-$ for all odd unicyclic components and $wt(C) \in \C \setminus \R^+$ for all even unicyclic components. Note that $c(H)$ denotes the number of connected components of $H$ and $U_i$ is a connected component of $H$ with cycle $C_i$.
   }
\end{theorem}

\begin{proof}
    \color{blue}{
    \textbf{Case 1: } $G$ is a tree.
    \\
    \indent Since $Q = NN^*$ we have that $\rank(Q) = \rank(NN^*) = \rank(N) = n-1$ by Theorem \ref{thm:rank}. Since $Q$ has dimension $n \times n$, this implies that $\det(Q) = 0$. Since $G$ is a tree, this means $H$ is empty and the empty sum is indeed $0$.
    \\
    \textbf{Case 2: } $G$ has a cycle.
    \\
    \indent Since $G$ has a cycle, it follows that $m \geq n$. Using the Cauchy-Binet Formula, we deduce:
    \[
        \det(Q) = \det(NN^*) = \sum_{\substack{S \subseteq \{1,2,..,m\} \\ |S| = n}}\det(N(:,S])\det(\overline{N(:,S]}) = \sum_{\substack{S \subseteq \{1,2,..,m\} \\ |S| = n}}|\det(N(:,S])|^2
    \]
    For all $S$, $N(:,S]$ corresponds to the signless incidence matrix of a spanning subgraph of $G$ containing all edges of the form $e_i$ where $i \in S$ which we denote as the subgraph $H$. We then let $N(:,S] = N_H$. 
    \\
    If $H$ contained a vertex of degree $0$, then performing a cofactor expansion across its corresponding row in $N_H$ gives $\det(N_H) = 0$. If $H$ contained at least one connected component that was a tree, then $N_H$ would have $\rank(N_H) < n$ and thus $\det(N_H) = 0$. If $H$ has no vertex of degree 0, no tree component, and contained at least one multi-cyclic connected component, then $H$ cannot have $n$ vertices and edges. Thus, $H$ must only contain unicyclic components and if there exists an odd unicyclic component with negative real weight or if there exists an even unicyclic component with positive real weight, then because $\rank(N_H) < n$, we have $\det(N_H) = 0$. Therefore, only when $H$ is a spanning subgraph of $G$ on $n$ vertices and $n$ edges where the connected components of $H$ are all unicyclic graphs such that $wt(C) \in \C \setminus \R^-$ for all odd unicyclic components and $wt(C) \in \C \setminus \R^+$ for all even unicyclic components is $\det(N_H)$ non-trivial. 
    \\
    Since $N_H = PN'Q$ where $P,Q$ are permutation matrices and $N'$ is a block diagonal matrix of the signless incidence matrix of each unicyclic component, $U_i$, we get:
    $$|\det(N_H)|^2 = \prod_{i=1}^{c(H)} |\det(N_{U_i})|^2$$
    Using Theorom \ref{volformulas} (b) and the fact that $\vol^2(N_{U_i}) = |\det(N_{U_i})|^2$, we get $$|\det(N_{U_i})|^2 = 2|W_{U_i\setminus C_i}|
    \cdot \alpha(C_i)
    $$
    Thus, in summary we have:
    $$\det(Q) = \sum_{H \in \mathcal{U}(G)}  \prod_{i=1}^{c(H)} 2|W_{U_i\setminus C_i}|
    \cdot \alpha(C_i)$$
    $$
   \alpha(C_i) = \begin{cases}
       |wt(C_i)| - \Re(wt(C_i))  & \text{if $C_i$ is an even cycle} \\
       |wt(C_i)| + \Re(wt(C_i))  & \text{if $C_i$ is an odd cycle} 
   \end{cases}
   $$
    }
\end{proof}

\begin{theorem}\label{final3}
    \color{blue}{Let $G$ be a weakly connected oriented graph on $n \geq 2$ vertices $1,2,..n$ and $m \geq 1$ edges $1,2,..m$ with non-zero edge weights in $\C \setminus \R^-$ and complex signless Laplacian matrix $Q$ such that $G$ contains oriented cycles and $wt(C) \in \R^-$ for all odd cycles and $wt(C) \in \R^+$ for all even cycles. Then, 
    $$(Q^+)_{ij} = \Biggl[ \frac{1}{ \sum_{T_ \in \mathcal{T}(G)} |W_{T}|} \Biggr]^2\sum_{k=1}^m \sum_{\substack{T_r, T_s \in \mathcal{T}(G) \\ e_k \in T_r \cap T_s}}
    |W_{T_r}||W_{T_s}|\overline{(N^+_{T_r})_{ki}}(N^+_{T_s})_{kj}$$
    where the entries of $N^+_{T_r}$ and $N^+_{T_s}$ are given by:
    $$
        (N^+_T)_{ij} = \frac{(-1)^{d(e_i,j)}}{n}
    \begin{cases}
        \frac{|T_h(e_i)|}{\sqrt{w_{uv}}}\frac{|wt(\gamma_{j \rightarrow u})|}{wt(\gamma_{j \rightarrow u})} & \text{where } j \in T_t(e_i)
        \\
        \\
        \frac{|T_t(e_i)|}{\sqrt{\overline{w_{uv}}}}\frac{|wt(\gamma_{j \rightarrow v})|}{wt(\gamma_{j \rightarrow v})} & \text{where } j \in T_h(e_i) 
        \\
    \end{cases}
    $$
    }
\end{theorem}

\begin{proof}
    \color{blue}{Since $Q = NN^*$, we have $Q^+ = (N^+)^*N^+$. Therefore,
    $$
    (Q^+)_{ij} = \sum_{k=1}^m((N^+)^*)_{ik}(N^+)_{kj} = \sum_{k=1}^m\overline{(N^+)_{ki}}(N^+)_{kj}
    $$
    Using Theorem \ref{final1}, we then get:
    $$
    (Q^+)_{ij} = \sum_{k=1}^m \Biggl[ \frac{1}{ \sum_{T_r \in \mathcal{T}(G)} |W_{T_r}|}\sum_{T_r \in \mathcal{T}(G), e_k \in T_r}|W_{T_r}|\overline{(N_{T_r}^+)_{ki}} \times \frac{1}{ \sum_{T_s \in \mathcal{T}(G)} |W_{T_s}|}\sum_{T_s \in \mathcal{T}(G), e_k \in T_s}|W_{T_s}|(N_{T_s}^+)_{kj} \Biggr]
    $$
    $$
    = \Biggl[ \frac{1}{ \sum_{T_ \in \mathcal{T}(G)} |W_{T}|} \Biggr]^2 \sum_{k=1}^m\sum_{\substack{T_r, T_s \in \mathcal{T}(G) \\ e_k \in T_r \cap T_s}}
    |W_{T_r}||W_{T_s}|\overline{(N^+_{T_r})_{ki}}(N^+_{T_s})_{kj}
    $$
    }
\end{proof}

\begin{theorem}\label{final4}
   \color{blue}{Let $G$ be a weakly connected oriented graph on $n \geq 2$ vertices $1,2,..n$ and $m \geq 1$ edges $1,2,..m$ with non-zero edge weights in $\C \setminus \R^-$ and complex signless Laplacian matrix $Q$ such that $G$ contains oriented cycles and there exists $wt(C) \in \C \setminus \R^-$ for an odd cycle or $wt(C) \in \C \setminus \R^+$ for an even cycle. Then,
   $$(Q^+)_{ij} = \frac{1}{\vol^4(N)}\sum_{k=1}^m \sum_{\substack{H_r,H_s \in 
   \mathcal{U}(G) \\ e_k \in H_r \cap H_s}} \prod_{t = 1}^{C(H_r)}2|W_{U_t \setminus C_t}| \cdot  \alpha(C_t) \prod_{z = 1}^{C(H_s)}2|W_{U_z \setminus C_z}| \cdot  \alpha(C_z)\overline{(N^+_{H_r})_{ki}}(N^+_{H_s})_{kj}$$
   where $\mathcal{U}(G)$ denotes the set of all spanning subgraphs $H$ of $G$ on $n$ edges consisting of $c(H)$ unicyclic components $U_1, U_2, ..., U_{c(H)}$ with $wt(C) \in \C \setminus \R^+$ for all even cycles and  $wt(C) \in \C \setminus \R^-$ for all odd cycles and when $e_i$ is in $U_r$ for some r. Also:
\\
    $$\alpha(C_i) = \begin{cases}
       |wt(C_i)| - \Re(wt(C_i))  & \text{if $C_i$ is an even cycle} \\
       |wt(C_i)| + \Re(wt(C_i))  & \text{if $C_i$ is an odd cycle} 
    \end{cases}
    $$
    and $(N_{H_r}^+)_{ij}$ and $(N_{H_s}^+)_{ij}$ are given by:
   \\
   If $e_i$ is in $U_r$ with even cycle $C_r$, then 
    $$(N^+_{H})_{ij} = 
    \begin{cases}
        (-1)^{d^*(j,u)}\frac{\frac{wt(\gamma_{u \rightarrow j})}{|wt(\gamma_{u \rightarrow j})|}}{\sqrt{\overline{w_{uv}}}(\frac{w_{uv}}{|w_{uv}|}-\frac{wt(\gamma_{u \rightarrow v})}{|wt(\gamma_{u \rightarrow v})|})} & \text{where } e_i = (u,v) \in C_r
        \\
        \\
        0 & \text{where } e_i \notin C \text{and } j \in G \setminus e_i[C_r] 
        \\
        \\
        (-1)^{d^*(j,u)}\frac{1}{\sqrt{{w_{uv}}}}\frac{|wt(\gamma_{j \rightarrow u})|}{wt(\gamma_{j \rightarrow u})} & \text{where } e_i \notin C \text{and } j \in G \setminus (e_i,u](C_r) 
        \\
        \\
        (-1)^{d^*(j,v)}\frac{1}{\sqrt{\overline{w_{uv}}}}\frac{|wt(\gamma_{j \rightarrow v})|}{wt(\gamma_{j \rightarrow v})} & \text{where } e_i \notin C \text{and } j \in G \setminus (e_i,v](C_r) 
        \\
        0 & \text{where } j \notin U_r
    \end{cases}
    $$
    If $e_i$ is in $U_r$ with odd cycle $C_r$, then
    $$(N^+_{H})_{ij} = 
    \begin{cases}
        (-1)^{d^*(j,u)}\frac{\frac{wt(\gamma_{u \rightarrow j})}{|wt(\gamma_{u \rightarrow j})|}}{\sqrt{\overline{w_{uv}}}(\frac{w_{uv}}{|w_{uv}|}+\frac{wt(\gamma_{u \rightarrow v})}{|wt(\gamma_{u \rightarrow v})|})} & \text{where } e_i = (u,v) \in C_r
        \\
        \\
        0 & \text{where } e_i \notin C \text{and } j \in G \setminus e_i[C_r] 
        \\
        \\
        (-1)^{d^*(j,u)}\frac{1}{\sqrt{{w_{uv}}}}\frac{|wt(\gamma_{j \rightarrow u})|}{wt(\gamma_{j \rightarrow u})} & \text{where } e_i \notin C \text{and } j \in G \setminus (e_i,u](C_r) 
        \\
        \\
        (-1)^{d^*(j,v)}\frac{1}{\sqrt{\overline{w_{uv}}}}\frac{|wt(\gamma_{j \rightarrow v})|}{wt(\gamma_{j \rightarrow v})} & \text{where } e_i \notin C \text{and } j \in G \setminus (e_i,v](C_r) 
        \\
        0 & \text{where } j \notin U_r
    \end{cases}
    $$
    and $\vol^2(N)$ is given by Theorem \ref{volformulas} (b). 
   }
\end{theorem}

\begin{proof}
    \color{blue}{Since $Q = NN^*$, we have $Q^+ = (N^+)^*N^+$. Therefore,
    $$
    (Q^+)_{ij} = \sum_{k=1}^m((N^+)^*)_{ik}(N^+)_{kj} = \sum_{k=1}^m\overline{(N^+)_{ki}}(N^+)_{kj}
    $$
    By Theorem \ref{final2}, we can then deduce:
    $$(Q^+)_{ij} = \sum_{k=1}^m \Biggl[\frac{1}{\vol^2(N)}\sum_{H_r \in \mathcal{U}(G), e_k \in H_r} \overline{(N_{H_r}^+)_{ki}} \prod_{t = 1}^{c(H_r)}2|W_{U_t\setminus C_t}|
    \cdot \alpha(C_t) 
    $$
    $$
    \times \frac{1}{\vol^2(N)}\sum_{H_s \in \mathcal{U}(G), e_k \in H_s}{(N_{H_s}^+)_{kj}} \prod_{z = 1}^{c(H_s)}2|W_{U_z\setminus C_z}|
    \cdot \alpha(C_z) \Biggr]
    $$
    $$= \frac{1}{\vol^4(N)}\sum_{k=1}^m \sum_{\substack{H_r,H_s \in 
   \mathcal{U}(G) \\ e_k \in H_r \cap H_s}} \prod_{t = 1}^{C(H_r)}2|W_{U_t \setminus C_t}| \cdot  \alpha(C_t) \prod_{z = 1}^{C(H_s)}2|W_{U_z \setminus C_z}| \cdot  \alpha(C_z)\overline{(N^+_{H_r})_{ki}}(N^+_{H_s})_{kj}
    $$
    }   
\end{proof}

\section*{Acknowledgment}
We thank Professor Tom Cuchta of Marshall University for organizing the Appalachian Mathematics and Physics Site REU, during which we conducted this research, and the National Science Foundation for funding this research under NSF REU Grant No. $2349289$.


\newpage
\begin{thebibliography}{99}
\bibitem{Milica1} A. Alazemi, O. Alhalabi, and M. Andeli\'c, Moore--Penrose inverse of the signless Laplacians of bipartite graphs, {\em Bull. Iran. Math. Soc.} 49, 51 (2023). 

\bibitem{Milica2} A. Alazemi, M. Andelic, and S. Mallik, Signed graphs and inverses of their incidence matrices, {\em Linear Algebra Appl.} 694 (2024) 78--100.


\bibitem{AB} A. Azimi and R. B. Bapat, Moore--Penrose inverse of the incidence matrix of a distance regular graph, {\em Linear Algebra Appl.} 551 (2018) 92--103.

\bibitem{ABE} A. Azimi, R. B. Bapat, and E. Estaji, Moore--Penrose inverse of incidence matrix of graphs with complete and cyclic blocks, {\em Discrete Math.} 342 (2019) 10--17.

\bibitem{graphs and matrices} R. B. Bapat, {\it Graphs and matrices},  Springer, 2014.

\bibitem{bap} R. B. Bapat, Moore--Penrose inverse of the incidence matrix of a tree, {\em Linear  Multilinear Algebra}. 42 (1997) 159--167.

\bibitem{MP2026} S. Barik, S. Mallik, and S. Reddy, Moore-Penrose inverse of the complex Laplacian of an oriented graph, {\em Discrete Math.} 349 (2026) 115117. 

\bibitem{SUG} S. Barik, S. U. Reddy, and  G. Sahoo, On singularity and properties of eigenvectors of complex Laplacian matrix of multidigraphs, {\em AKCE Int. J. Graphs Comb.} 20(2) (2023) 125--133.
\bibitem{SUG2} S. Barik and S. U. Reddy, {\it On the $A_{\mathbb{C}}$-rank of multidigraphs}, Aequat. Math. 98 (2024) 189--213.

\bibitem{SG} S. Barik and  G. Sahoo, A new matrix representation of multidigraphs, {\em AKCE Int. J. Graphs Comb.} 17(1) (2020) 466--479.


\bibitem{Ben-Israel} A. Ben-Israel, A volume associated with $m\times n$ matrices, {\em Linear Algebra Appl.} 167 (1992) 87--111.

\bibitem{BG} A. Ben-Israel and T. N. E. Greville, {\em Generalized Inverses: Theory and Applications}, Wiley-Interscience, 1974.


\bibitem{CRC}  D. Cvetkovi\' c, P. Rowlinson, and S. K. Simi\' c, Signless Laplacians of finite graphs,{\em Linear Algebra Appl.} 423 (2007) 155--171.


\bibitem{Harary} F. Harary, The number of oriented graphs, {\em Michigan Math. J.} 4(3) (1957) 221--224. DOI: 10.1307/mmj/1028997952.

\bibitem{HM} K. Hassani Monfared and S. Mallik, An analog of Matrix Tree Theorem for signless Laplacians, {\em Linear Algebra Appl.} 560 (2019) 43--55.

\bibitem{HM1} R. Hessert and S. Mallik, Moore--Penrose Inverses of the Signless Laplacian and Edge-Laplacian of Graphs, {\em Discrete Math.}  344 (2021) 112451.

\bibitem{HM2} R. Hessert and S. Mallik, The inverse of the incidence matrix of a unicyclic graph, {\em Linear Multilinear Algebra.} 71(4) (2023) 513--527.

\bibitem{I} Y. Ijiri, On the generalized inverse of an incidence matrix, {\em Jour. Soc. Indust. Appl. Math.} 13(3) (1965) 827--836 .

\bibitem{Ipsen} J. Ipsen and S. Mallik, Incidence and Laplacian matrices of wheel graphs and their inverses,  {\em Am. J. Comb.} 2 (2023) 1--19.


\bibitem{MallikReddyVolume} S. Mallik and S. U. Reddy, Volume of a complex matrix and its applications in oriented graphs, {\it Matematički Vesnik} 78(2) (2026) (to appear).  

\end{thebibliography}
\end{document}